\documentclass[a4paper,11pt]{amsart}
 \usepackage{}
  \usepackage{color}
 \usepackage{amsfonts}
 \usepackage{amsfonts}
 \usepackage{amsfonts}
 \usepackage{bbm}
 \usepackage{amsmath}
 \usepackage{amssymb}
 \usepackage{amsfonts}
 \usepackage{mathrsfs}
 \usepackage[all]{xy}
 \usepackage{amsthm,amsmath,amsfonts,amssymb}
\newtheorem{thm}{Theorem}[section]
\newtheorem{lem}{Lemma}[section]

\newtheorem{cor}{Corollary}[section]
\newtheorem{pro}{Proposition}[section]
\newtheorem{rem}{Remark}[section]

\theoremstyle{definition}
\def\AC{\tau(A)}

\begin{document}

\numberwithin{equation}{section}

\title[    closed positive currents  ]
 {\bf  On quaternionic Monge-Amp\`{e}re operator, closed positive currents and Lelong-Jensen type formula on quaternionic space}
\author{Dongrui Wan and Wei Wang}
\thanks{
Supported by National Nature Science Foundation in China (No.
11171298)\\ College of Mathematics and Computational Science, Shenzhen University, Shenzhen, 518060, P. R. China, Email: wandongrui@szu.edu.cn; Department of Mathematics,
Zhejiang University, Zhejiang 310027,
 P. R. China, Email:   wwang@zju.edu.cn}

\begin{abstract}In this paper, we introduce the first-order differential operators $d_0$ and $d_1$ acting on the quaternionic version of differential forms on the flat quaternionic space $\mathbb{H}^n$. The behavior of $d_0,d_1$ and $\triangle=d_0d_1$ is very similar to $\partial,\overline{\partial}$ and $\partial \overline{\partial}$ in several complex variables. The quaternionic Monge-Amp\`{e}re operator can be defined as $(\triangle u)^n$ and has a simple  explicit expression. We define the notion of closed positive currents in the quaternionic case, and extend several results in complex pluripotential theory to the quaternionic case: define the Lelong number for closed positive currents, obtain the quaternionic version of Lelong-Jensen type formula, and generalize   Bedford-Taylor theory, i.e., extend the definition of the quaternionic Monge-Amp\`{e}re operator to   locally bounded quaternionic plurisubharmonic functions and prove the corresponding convergence theorem.
\end{abstract}
\maketitle
\begin{center}
\begin{minipage}{135mm}
\end{minipage}
\end{center}
\section{Introduction}\par
Recently, people are interested in developing pluripotential theory on quaternionic manifold \cite{alesker1}-\cite{alesker6},\cite{alesker8},\cite{wang1} and more generally, on calibrated manifolds \cite{harvey2}-\cite{harvey3}. The quaternionic Monge-Amp\`{e}re operator is defined as the Moore determinant of the quaternionic Hessian of $u$:
 \begin{equation*} det\left[\frac{\partial^2u}{\partial q_j\partial \bar{q}_k}(q)\right].\end{equation*}Alesker proved in \cite{alesker1} a quaternionic version of Chern-Levine-Nirenberg estimate and extended the definition of quaternionic Monge-Amp\`{e}re operator to continuous quaternionic plurisubharmonic functions. The quaternionic Monge-Amp\`{e}re operator on hypercomplex manifolds was introduced by Alesker and Verbitsky \cite{alesker6}, which coincides with the above form when the manifold is flat.

To define the quaternionic Monge-Amp\`{e}re operator on general quaternionic manifolds, Alesker introduced in \cite{alesker2} an operator in terms of the Baston operator $\triangle$,   the first operator of the quaternionic complex on quaternionic manifolds. This operator is exactly the quaternionic Monge-Amp\`{e}re operator when the manifold is flat. Moreover, he used methods in complex geometry (twistor transformation and spectral sequences, etc.) to prove the following multiplicative property of the Baston operator:
\begin{equation}\label{1.11}\triangle(\omega\wedge\triangle\eta)=\triangle\omega\wedge\triangle\eta,\end{equation}where $\omega,\eta$ are sections of certain  bundles. The quaternionic version of Chern-Levine-Nirenberg estimate follows from (\ref{1.11}) directly.

The quaternionic complex on the flat space $\mathbb{H}^n$ is also called $k$-Cauchy-Fueter complex in \cite{Wang}. The first operator of this complex is $k$-Cauchy-Fueter operator, whose kernel consists of $k$-regular function. The Baston operator $\triangle$ is the first operator of $0$-Cauchy-Fueter complex:\begin{equation}\label{1.1}0\rightarrow C^\infty(\Omega,\mathbb{C})\xrightarrow{\triangle}C^\infty(\Omega,\wedge^2\mathbb{C}^{2n})
\xrightarrow{D}C^\infty(\Omega,\wedge^3\mathbb{C}^{2n})\rightarrow\cdots.
\end{equation}In order to study   $k$-regular functions of several quaternionic variables, the second author \cite{Wang} wrote down explicitly each operator of the $k$-Cauchy-Fueter complex in terms of real variables. Therefore we have an explicit expression of the quaternionic Monge-Amp\`{e}re operator, which is much more convenient than the definition by using  Moore determinant. Motivated by this formula, we introduce two first-order differential operators $d_0,d_1$, whose behavior is very similar to $\partial $ and $ \overline{\partial}$ in complex pluripotential theory, and write the operator $\triangle$ as $d_0d_1$. Based on this observation, we can establish the quaternionic versions of several results in the complex pluripotential theory.

To write the quaternionic Monge-Amp\`{e}re operator in terms of real variables, we will use the well known embedding of the quaternionic algebra $\mathbb{H}$ into End$(\mathbb{C}^{2})$ defined by
\begin{equation}\label{2.1}x_0+x_1\textbf{i}+x_2\textbf{j}+x_3\textbf{k}\mapsto\left(
                                                                                 \begin{array}{cc}
                                                                                   x_0+\textbf{i}x_1 & -x_2-\textbf{i}x_3 \\
                                                                                   x_2-\textbf{i}x_3 & x_0-\textbf{i}x_1 \\
                                                                                 \end{array}
                                                                               \right).
\end{equation}
Actually we will use the conjugate embedding
\begin{equation}\label{2.2}\begin{aligned}\tau:\mathbb{H}^{n}\cong\mathbb{R}^{4n}&\hookrightarrow\mathbb{C}^{2n\times2},\\ (q_0,\ldots,q_{n-1})&\mapsto \textbf{z}=(z^{j\alpha})\in\mathbb{C}^{2n\times2},
\end{aligned}\end{equation}
$q_j=x_{4j}+\textbf{i}x_{4j+1}+\textbf{j}x_{4j+2}+\textbf{k}x_{4j+3}$, $j=0,1,\ldots,2n-1, ~\alpha=0 ,1 ,$ with
\begin{equation}\label{2.3}\left(
                             \begin{array}{cc}
                               z^{00 } & z^{01 } \\
                               z^{10 } & z^{11 } \\
                               \vdots&\vdots\\
                               z^{(2l)0 } & z^{(2l)1 } \\
                               z^{(2l+1)0 } & z^{(2l+1)1 } \\
                               \vdots&\vdots\\
                               z^{(2n-2)0 } & z^{(2n-2)1 } \\
                               z^{(2n-1)0 } & z^{(2n-1)1 } \\
                             \end{array}
                           \right):=\left(
                                      \begin{array}{cc}
                                      x_{0}-\textbf{i}x_{1} & -x_{2}+\textbf{i}x_{3} \\
                                        x_{2}+\textbf{i}x_{3} & x_{0}+\textbf{i}x_{1} \\
                                        \vdots&\vdots\\
                                         x_{4l}-\textbf{i}x_{4l+1} & -x_{4l+2}+\textbf{i}x_{4l+3} \\
                                        x_{4l+2}+\textbf{i}x_{4l+3} & x_{4l}+\textbf{i}x_{4l+1} \\
                                        \vdots&\vdots\\
                                        x_{4n-4}-\textbf{i}x_{4n-3} & -x_{4n-2}+\textbf{i}x_{4n-1} \\
                                        x_{4n-2}+\textbf{i}x_{4n-1} & x_{4n-4}+\textbf{i}x_{4n-3} \\
                                      \end{array}
                                    \right).
\end{equation} Pulling back to the quaternionic
space $\mathbb{H}^n\cong\mathbb{R}^{4n}$ by the embedding
(\ref{2.3}), we define on $\mathbb{R}^{4n}$ first-order differential operators $\nabla_{j\alpha}$ as following:
\begin{equation}\label{2.4}\left(
                             \begin{array}{cc}
                               \nabla_{00 } & \nabla_{01 } \\
                               \nabla_{10 } & \nabla_{11 } \\
                               \vdots&\vdots\\
                               \nabla_{(2l)0 } & \nabla_{(2l)1 } \\
                               \nabla_{(2l+1)0 } & \nabla_{(2l+1)1 } \\
                               \vdots&\vdots\\
                               \nabla_{(2n-2)0 } & \nabla_{(2n-2)1 } \\
                               \nabla_{(2n-1)0 } & \nabla_{(2n-1)1 } \\
                             \end{array}
                           \right):=\left(
                                      \begin{array}{cc}
                                      \partial_{x_{0}}+\textbf{i}\partial_{x_{1}} & -\partial_{x_{2}}-\textbf{i}\partial_{x_{3}} \\
                                        \partial_{x_{2}}-\textbf{i}\partial_{x_{3}} & \partial_{x_{0}}-\textbf{i}\partial_{x_{1}} \\
                                         \vdots&\vdots\\
                                         \partial_{x_{4l}}+\textbf{i}\partial_{x_{4l+1}} & -\partial_{x_{4l+2}}-\textbf{i}\partial_{x_{4l+3}} \\
                                        \partial_{x_{4l+2}}-\textbf{i}\partial_{x_{4l+3}} & \partial_{x_{4l}}-\textbf{i}\partial_{x_{4l+1}} \\
                                        \vdots&\vdots\\
                                        \partial_{x_{4n-4}}+\textbf{i}\partial_{x_{4n-3}} & -\partial_{x_{4n-2}}-\textbf{i}\partial_{x_{4n-1}} \\
                                        \partial_{x_{4n-2}}-\textbf{i}\partial_{x_{4n-1}} & \partial_{x_{4n-4}}-\textbf{i}\partial_{x_{4n-3}} \\
                                      \end{array}
                                    \right).
\end{equation}
The advantage of using these operators is that $\nabla_{j\alpha}z^{k\beta}=2\delta_j^{k}\delta_{\alpha}^{\beta}$ (cf. Lemma \ref{l3.1}), i.e., $z^{k\beta}$'s can be viewed as independent variables and $\nabla_{j\alpha}$'s are derivatives with respect to these variables. The operators $\nabla_{j\alpha}$'s play very important roles in the investigating of regular functions in several quaternionic variables \cite{kang} \cite{Wang}. The Baston operator is given by the determinants of $(2\times2)$-submatrices of (\ref{2.4}).
\par

Let $\wedge^{2k}\mathbb{C}^{2n}$ be the complex exterior algebra generated by $\mathbb{C}^{2n}$, $0\leq k\leq n$. Fix a basis
$\{\omega^0,\omega^1,\ldots$, $\omega^{2n-1}\}$ of $\mathbb{C}^{2n}$. Let $\Omega$ be a domain in $\mathbb{R}^{4n}$. We
define $d_0,d_1:C_0^\infty(\Omega,\wedge^{p}\mathbb{C}^{2n})\rightarrow C_0^\infty(\Omega,\wedge^{p+1}\mathbb{C}^{2n})$ by \begin{equation}\begin{aligned}\label{2.228}&d_0F=\sum_{k,I}\nabla_{k0 }f_{I}~\omega^k\wedge\omega^I,\\
&d_1F=\sum_{k,I}\nabla_{k1 }f_{I}~\omega^k\wedge\omega^I,\\
&\triangle
F=d_0d_1F,
\end{aligned}\end{equation}for $F=\sum_{I}f_{I}\omega^I\in C_0^\infty(\Omega,\wedge^{p}\mathbb{C}^{2n})$,  where the multi-index
$I=(i_1,\ldots,i_{p})$ and
$\omega^I:=\omega^{i_1}\wedge\ldots\wedge\omega^{i_{p}}$. The operators $d_0$ and $d_1$ depend on the choice of the coordinates $x_j$'s and the basis $\{\omega^j \}$. It is known (cf. \cite{Wang}) that the second operator $D$ in the $0$-Cauchy-Fueter complex (\ref{1.1}) can be written as $DF:=\left(
                                                                                                                               \begin{array}{c}
                                                                                                                                 d_0F \\
                                                                                                                                 d_1F \\
                                                                                                                               \end{array}
                                                                                                                             \right)$.
Although $d_0,d_1$ are not exterior differential, their behavior is similar to the exterior differential: $d_0d_1=-d_1d_0$; $d_0^2=d_1^2=0$; for $F\in C_0^\infty(\Omega,\wedge^{p}\mathbb{C}^{2n})$, $G\in C_0^\infty(\Omega,\wedge^{q}\mathbb{C}^{2n})$, we have\begin{equation}\label{da}d_\alpha(F\wedge G)=d_\alpha F\wedge G+(-1)^{p}F\wedge d_\alpha G, \end{equation} $\alpha=0,1$, and \begin{equation}\label{1.2}d_0\triangle=d_1\triangle=0.
\end{equation}  (\ref{1.1}) is a complex since $D\triangle=0$.

We say $F$
is \emph{closed} if
\begin{equation}\label{2.229}d_0F=d_1F=0,~~~~i.e.,~~~~DF=0.
\end{equation}We prove that  for $u_1,\ldots,
u_n\in C^2$, $\triangle
u_1\wedge\ldots\wedge\triangle u_k$ is closed, $k=1,\ldots,n $. Moreover, it follows easily from (\ref{da}) and (\ref{1.2}) that $\triangle u_1\wedge\ldots\wedge\triangle u_n$ satisfies the following remarkable identities:\begin{equation}\label{2.37}\begin{aligned}\triangle u_1\wedge \triangle
u_2\wedge\ldots\wedge\triangle u_n&=d_0(d_1u_1\wedge \triangle
u_2\wedge\ldots\wedge\triangle u_n)\\&=d_0d_1(u_1\triangle
u_2\wedge\ldots\wedge\triangle u_n)=\triangle (u_1
\triangle u_2\wedge\ldots\wedge\triangle u_n).
\end{aligned}\end{equation}This is an improved version of Alesker's identity (\ref{1.11}).
Denote by $\wedge^{2k}_{\mathbb{R}}\mathbb{C}^{2n}$ the subspace of all real elements in $\wedge^{2k}\mathbb{C}^{2n}$ following Alesker \cite{alesker2}. They are counterparts of $(k,k)$-forms in several complex variables. In the space $\wedge^{2k}_{\mathbb{R}}\mathbb{C}^{2n}$ we define convex cones $\wedge^{2k}_{\mathbb{R}+}\mathbb{C}^{2n}$ and $SP^{2k}\mathbb{C}^{2n}$ of positive and strongly positive elements, respectively.

Denoted by $\mathcal {D}^{2k}(\Omega)$ the set of all $C_0^\infty(\Omega)$ functions   valued in $\wedge^{2k}\mathbb{C}^{2n}$.
$\eta\in\mathcal{D}^{2k}(\Omega)$ is called a \emph{positive form} (respectively, \emph{strongly positive form}) if for any $q\in\Omega$, $\eta(q)$ is a positive (respectively, strongly positive) element. Such forms are the same as the sections of certain line bundle introduced by Alesker \cite{alesker2} when the manifold is flat. But our definition of closedness is new.

Now recall that an upper semicontinuous function $u$ on
$\mathbb{H}^{n}$ is said to be \emph{plurisubharmonic} if $u$ is
subharmonic on each right quaternionic line. Denote by $PSH$ the class
of all plurisubharmonic functions (cf.
\cite{alesker1,alesker3,alesker2} for more information about
plurisubharmonic functions). We prove that for $u\in PSH\cap C^2(\Omega)$, $\triangle u$ is a closed strongly positive
$2$-form.

An element of the dual space
$(\mathcal {D}^{2n-p}(\Omega))' $ is called a~\emph{$p$-current}. A $2k$-current $T$ is said to be \emph{positive} if we have $T(\eta)\geq0$ for any strongly positive form $\eta\in\mathcal {D}^{2n-2k}(\Omega)$.
Although a $2n$-form is not an authentic differential form and we cannot integrate it, we can define \begin{equation}\label{2.24}\int_\Omega F:=\int_\Omega f dV,
\end{equation}if we write $F=f~\Omega_{2n}\in L^1(\Omega,\wedge^{2n}\mathbb{C}^{2n})$,
where $dV$ is the Lesbesgue measure and \begin{equation}\label{2.41}\Omega_{2n}:=\omega^0\wedge
\omega^1\wedge\ldots\wedge\omega^{2n-2}\wedge
\omega^{2n-1}.\end{equation}
For a $2n$-current $F=\mu~\Omega_{2n}$ with the coefficient to be measure $\mu$, define
\begin{equation}\label{2.273}\int_\Omega F:=\int_\Omega \mu.\end{equation}

Now for the $p$-current $F$, we define $d_\alpha F$ as
\begin{equation}\label{2.271}(d_\alpha F)(\eta):=-F(d_\alpha\eta),\qquad \alpha=0,1,
\end{equation}for any test $(2n-p-1)$-form $\eta$. We say a current $F$ is \emph{closed} if
\begin{equation}\label{3.30}d_0F=d_1F=0,~i.e.,~DF=0.\end{equation}$\triangle u$ is a closed positive $2$-current for any $u\in PSH(\Omega)$.

Quaternionic positive currents are discussed in \cite{alesker2} \cite{harvey4}, while closed current is discussed in \cite{alesker6} by using the usual exterior differentiation. Our definition of closedness matches positivity well, and several results of closed positive currents in several complex variables can be extended to the quaternionic case. As in the complex case, in general we cannot define the wedge product of two currents, but we can generalize Bedford-Taylor  theory \cite{bed} in   complex  analysis to our case. Let $u$ be a locally bounded $PSH$ function and let $T$ be a closed
positive $2k$-current. Define\begin{equation*}\triangle u\wedge
T:=\triangle(uT),\end{equation*}i.e., $(\triangle u\wedge
T)(\eta):=uT(\triangle\eta)$ for   test form $\eta$. We show that $\triangle u\wedge T$ is also a closed positive
current. Inductively,
\begin{equation*}\triangle
u_1\wedge\ldots\wedge\triangle u_p :=\triangle(u_1\triangle
u_2\ldots\wedge\triangle u_p )\end{equation*} is a closed
positive $2p$-current, when $u_1,\ldots,u_p\in PSH\cap
L_{loc}^\infty(\Omega)$.

Set
\begin{equation}\label{bn}\beta_n:=\sum_{l=0}^{n-1}\omega^{2l}\wedge\omega^{2l+1}\in\wedge^{2}_{\mathbb{R}+}\mathbb{C}^{2n}.\end{equation}
For a closed
positive $(2n-2p)$-current $T$,  $T\wedge\beta_n^p$ a closed
positive $ 2n $-current. Define
\begin{equation}\label{5.7}\sigma_T(a,r):=\int_{B(a,r)}T\wedge\beta_n^p\end{equation}for small $r$, $a\in\Omega$. It can be shown that
$\frac{\sigma_{T}(a,r)}{r^{4p}}$ is an increasing function of $r$. So we can introduce the \emph{Lelong number of $T$ at point $a$} as the limit
\begin{equation}\label{5.8}\nu_a(T):=\lim_{r\rightarrow0+}\frac{\sigma_{T}(a,r)}{r^{4p}}.\end{equation}

Now let $\Omega$ be a \emph{quaternionic strictly pseudoconvex domain}, i.e. that $\Omega$
has a strictly $PSH$ exhaustion function. Let $\varphi$ be a continuous
$PSH$ function on $\Omega$. Denote $$B_\varphi(r):=\{q\in\Omega;\varphi(q)<r\}\quad\text{and}\quad S_\varphi(r):=\{q\in\Omega;\varphi(q)=r\}.$$
The \emph{quaternionic boundary measure} associated to
$\varphi$ is the nonnegative Borel measure $\mu_{\varphi,r}$ defined
as
\begin{equation}\label{4.1}\mu_{\varphi,r}=\triangle_n (\varphi_r)-\chi_{\Omega\backslash
B_\varphi(r)}\triangle_n \varphi,
\end{equation}where $\varphi_r:=\max\{\varphi,r\}$,  and $\triangle_n \varphi$ is the
coefficient of the $2n$-current $(\triangle\varphi)^n$, i.e.
\begin{equation*}
     (\triangle\varphi)^n=\triangle_n \varphi~\Omega_{2n} .
\end{equation*}
The measure $\mu_{\varphi,r}$ supports on
$S_\varphi(r)$. At last we prove the Lelong-Jensen type formula:\begin{equation*}\mu_{\varphi,r}(V)-\int_{B_\varphi(r)}V(\triangle\varphi)^n
 =\int_{-\infty}^{r}dt\int_{B_\varphi(t)}\triangle V\wedge(\triangle\varphi)^{n-1}.
\end{equation*}
The Lelong-Jensen formula, Lelong number and boundary measure play very important roles in complex pluripotential theory (cf. \cite{Demailly1987 ,Demailly1991,demailly}).

The quaternionic Monge-Amp\`{e}re operator can be written as
\begin{equation}\label{0.1}\triangle_nu=\sum_{i_1,j_1,\ldots }\delta^{i_1j_1\ldots
i_nj_n}_{01\ldots(2n-1)}\nabla_{i_10}\nabla_{j_11}u\ldots\nabla_{i_n0}\nabla_{j_n1}u,\end{equation}
where \begin{equation}\label{2.111}\delta^{i_1j_1\ldots
i_nj_n}_{01\ldots(2n-1)}:=
          \text{the~ sign~of~the~permutation~from}\,(i_1,j_1,\ldots
i_n,j_n)\,\text{to}\,(0,1,\ldots,2n-1),
                   \end{equation}
    if $\{i_1,j_1,\ldots,
i_n,j_n\}=\{0,1,\ldots,2n-1\}$; otherwise, $\delta^{i_1j_1\ldots
i_nj_n}_{01\ldots(2n-1)}=0$. We have the remarkable identities:\begin{equation}\label{0.2}\begin{aligned}\triangle_nu&=\sum_{i_1,j_1,\ldots }\nabla_{i_10}\left[\delta^{i_1j_1\ldots
i_nj_n}_{01\ldots(2n-1)}\nabla_{j_11}u\ldots\nabla_{i_n0}\nabla_{j_n1}u\right]\\
&=\sum_{i_1,j_1,\ldots }\nabla_{i_10}\nabla_{j_11}\left[\delta^{i_1j_1\ldots
i_nj_n}_{01\ldots(2n-1)}u\nabla_{i_20}\nabla_{j_21}u\ldots\nabla_{i_n0}\nabla_{j_n1}u\right].\end{aligned}\end{equation}
These identities are equivalent to (\ref{2.37}), and are very similar to that for real $k$-Hessian operators (cf. (2.9) in \cite{wan2}). Based on (2.9) in \cite{wan2},  we established several results in pluripotential theory for $k$-convex functions and $k$-Hessian operators in \cite{wan2}. We proved the identities (\ref{0.2}) first and realized that we can extend many results in pluripotential theory to the quaternionic case as in \cite{wan2}. Later, we found the simplified version of these results by introducing the quaternionic version of differential forms. Furthermore, these forms allow us to develop the theory of closed positive currents in the quaternionic case.

The paper is organized as
follows. The propositions on quaternionic linear algebra we need are collected in Section 2.1, and we establish useful properties of the operators $d_0,d_1$ and the Baston operator $\triangle$. In section 3, we define the notions of closed positive forms and closed positive currents, and prove that $\triangle u$ is a closed positive $2$-current for any $PSH$ function $u$. And we show that when functions $u_1,\ldots,u_k$ are locally bounded, $\triangle u_1\wedge\ldots\wedge\triangle u_k $ is a well defined closed positive current and is continuous in decreasing sequences. In Section 4 we introduce the Lelong number of closed positive currents. In the last section, we discuss the quaternionic boundary measure and establish the quaternionic version of Lelong-Jensen type formula. In Appendix A, we give an elementary proof of the coincidence of $\triangle_n$ in (\ref{0.1}) with the quaternionic Monge-Amp\`{e}re operator, which was proved by Alesker by an abstract method (Proposition 7.1 in \cite{alesker2}).

 \section{the operators $d_0,d_1$ and the Baston operator $\triangle$}\par
\subsection{The complex matrix associated to a quaternionic matrix}
We begin with the fact that the quaternionic algebra is isomorphic to a subalgebra of complex $(2\times2)$-matrices. Recall the conjugate embedding $\tau:\mathbb{H}\hookrightarrow\mathbb{C}^{2\times2}$ given by (\ref{2.3}).
\begin{lem}\label{l2.1}$(1)$
$\tau(q_0q_1)=\tau(q_0)\tau(q_1),$ for $q_0,q_1\in \mathbb{H}$.\\
$(2)$ $\tau(\overline{q_0})=\overline{\tau(q_0)}^t$, for $q_0\in
\mathbb{H}$.
\end{lem}
\begin{proof}Let $q_0=x_0+\textbf{i}x_1+\textbf{j}x_2+\textbf{k}x_3$
and $q_1=x_4+\textbf{i}x_5+\textbf{j}x_6+\textbf{k}x_7$. Then
\begin{equation*}\begin{aligned}q_0q_1=&[x_0+\textbf{i}x_1+\textbf{j}(x_2-\textbf{i}x_3)][x_4+\textbf{i}x_5+\textbf{j}(x_6-\textbf{i}x_7)]\\
=&(x_0+\textbf{i}x_1)(x_4+\textbf{i}x_5)-(x_2+\textbf{i}x_3)(x_6-\textbf{i}x_7)\\
&\qquad+\textbf{j}[(x_0-\textbf{i}x_1)(x_6-\textbf{i}x_7)+(x_2-\textbf{i}x_3)(x_4+\textbf{i}x_5)].
\end{aligned}\end{equation*}It follows that $\tau(q_0q_1)=$\begin{equation*}\begin{aligned}& \left(
                                                                           \begin{array}{cc}
                                                                             (x_0-\textbf{i}x_1)(x_4-\textbf{i}x_5)-(x_2-\textbf{i}x_3)(x_6+\textbf{i}x_7) & -(x_0-\textbf{i}x_1)(x_6-\textbf{i}x_7)-(x_2-\textbf{i}x_3)(x_4+\textbf{i}x_5) \\
                                                                             (x_0+\textbf{i}x_1)(x_6+\textbf{i}x_7)+(x_2+\textbf{i}x_3)(x_4-\textbf{i}x_5) & (x_0+\textbf{i}x_1)(x_4+\textbf{i}x_5)-(x_2+\textbf{i}x_3)(x_6-\textbf{i}x_7) \\
                                                                           \end{array}
                                                                         \right)\\
=&\left(
    \begin{array}{cc}
     x_{0}-\textbf{i}x_{1} & -x_{2}+\textbf{i}x_{3} \\
                                        x_{2}+\textbf{i}x_{3} & x_{0}+\textbf{i}x_{1} \\
    \end{array}
  \right)\left(
           \begin{array}{cc}
            x_{4}-\textbf{i}x_{5} & -x_{6}+\textbf{i}x_{7} \\
                                        x_{6}+\textbf{i}x_{7} & x_{4}+\textbf{i}x_{5} \\
           \end{array}
         \right)=\tau(q_0)\tau(q_1).
\end{aligned}\end{equation*}And
$\overline{q_0}=x_0-\textbf{i}x_1-\textbf{j}x_2-\textbf{k}x_3$,
\begin{equation*}\tau(\overline{q_0})=\left(
    \begin{array}{cc}
     x_{0}+\textbf{i}x_{1} & x_{2}-\textbf{i}x_{3} \\
                                        -x_{2}-\textbf{i}x_{3} & x_{0}-\textbf{i}x_{1} \\
    \end{array}
  \right)= \overline{\left(\begin{array}{cc}
     x_{0}-\textbf{i}x_{1} & x_{2}+\textbf{i}x_{3} \\
                                        -x_{2}+\textbf{i}x_{3} & x_{0}+\textbf{i}x_{1} \\
    \end{array}
  \right)}=\overline{\tau(q_0)}^t.
\end{equation*}\end{proof}
Now we extend the definition of $\tau$ to a mapping from quaternionic $(l\times m)$-matrices to complex $(2l\times 2m)$-matrices. Let $A=(A_{jk})_{l\times m}$ be a quaternionic $(l\times m)$-matrix and write
$A_{jk}=a_{jk}^0+\textbf{i}a_{jk}^1+\textbf{j}a_{jk}^2+\textbf{k}a_{jk}^3\in\mathbb{H}$.
We define $\tau(A)$ is the complex $(2l\times 2m)$-matrix
$(\tau(A_{jk}))_{j=0,\ldots,l-1}^{k=0,\ldots,m-1},$ i.e.,\begin{equation}\label{tau}\tau(A)=\left(
                                                              \begin{array}{cccc}
                                                                \tau(A_{00}) & \tau(A_{01}) & \cdots & \tau(A_{0(m-1)}) \\
                                                                \tau(A_{10}) & \tau(A_{11}) & \cdots & \tau(A_{1(m-1)}) \\
                                                                \cdots & \cdots & \cdots & \cdots \\
                                                                \tau(A_{(l-1)0}) & \tau(A_{(l-1)1}) & \cdots & \tau(A_{(l-1)(m-1)}) \\
                                                              \end{array}
                                                            \right),
\end{equation}where $\tau(A_{jk})$ is the complex $(2\times2)$-matrix \begin{equation}\label{2.301}\left(
                                       \begin{array}{cc}
                                         a_{jk}^0-\textbf{i}a_{jk}^1 & -a_{jk}^2+\textbf{i}a_{jk}^3 \\
                                         a_{jk}^2+\textbf{i}a_{jk}^3 & a_{jk}^0+\textbf{i}a_{jk}^1 \\
                                       \end{array}
                                     \right).\end{equation}

Denote by $\text{GL}_\mathbb{H}(n)$ the set of all invertible quaternionic $(n\times
n)$-matrices, and denote by $\text{U}_{\mathbb{H}}(n)$ the set of all unitary quaternionic $(n\times
n)$-matrices, i.e., $\text{U}_{\mathbb{H}}(n)=\{A\in\text{GL}_\mathbb{H}(n),\overline{A}^tA=A\overline{A}^t=I_{n\times n}\}$, where $(\overline{A}^t)_{jk}=\overline{A_{kj}}$.
\begin{pro}\label{p2.1} $(1)$ $\tau(AB)=\tau(A)\tau(B)$ for a quaternionic $(p\times m)$-matrix $A$ and a quaternionic $(m\times l)$-matrix $B$. In particular, for $q'=Aq,$ $q,q' \in\mathbb{H}^n$, $A\in$ $\text{GL}_\mathbb{H}(n)$, we have
$\tau(q')=\tau(A)\tau(q)$ as complex $(2n\times2)$-matrix.\\
$(2)$ Let $A$ be in $\text{GL}_\mathbb{H}(n)$, then \begin{equation}\label{2.231}J\overline{\tau(A)}=\tau(A)J,\end{equation} where \begin{equation}\label{2.302}J=\left(
                                       \begin{array}{ccccccc}
                                       0 & 1 &   &   &   \\
                                         -1 & 0 &   &   &   \\
                                           &   & 0 & 1 &   \\
                                           &   & -1 & 0 &   \\
                                           &   &   &   & \ddots \\
                                           & & &&&0&1\\
                                           & & &&&-1&0\\\end{array}
                                     \right).
\end{equation}
$(3)$  $\tau\left(\overline{A}^t\right)=\overline{\tau(A)}^t$ for a quaternionic $(n\times n)$-matrix $A$. If $A\in\text{U}_{\mathbb{H}}(n)$, $\tau(A)$ is symplectic, i.e., $\tau(A)J\tau(A)^t=J$.
\end{pro}

\begin{proof}(1)
Let $C=AB$. By Lemma \ref{l2.1} (1),
\begin{equation*}\begin{aligned}&\tau(C_{jt})=\tau\left(\sum_{k=0}^{n-1}A_{jk}B_{kt}\right)=\sum_{k=0}^{n-1}\tau(A_{jk})\tau(B_{kt})\\&=
\left(\begin{array}{ccccc}
\tau(A_{j0})_{00} & \tau(A_{j0})_{01}&\ldots  &\tau(A_{j(n-1)})_{00}&\tau(A_{j(n-1)})_{01}\\
\tau(A_{j0})_{10} & \tau(A_{j0})_{11}&\ldots  &\tau(A_{j(n-1)})_{10}&\tau(A_{j(n-1)})_{11}\\
\end{array}
\right)\cdot\left(
\begin{array}{cc}
\tau(B_{0t})_{00}&\tau(B_{0t})_{01} \\
\tau(B_{0t})_{10}&\tau(B_{0t})_{11}\\
\vdots\\
\tau(B_{(n-1)t})_{00}&\tau(B_{(n-1)t})_{01} \\
\tau(B_{(n-1)t})_{10}&\tau(B_{(n-1)t})_{11}\\
\end{array}
\right).
\end{aligned}\end{equation*}Consequently, $\tau(C)=\tau(A)\tau(B).$ \par (2)
Denote $\tau(q ):=(z^{j\alpha}),$  $\tau(q'):=(w^{j\alpha}),$ $j=0,1,\ldots,2n-1,~\alpha=0 ,1 .$ Then we   have the equation
$\left(
    \begin{array}{c}
      \vdots \\
      w^{j0 } \\
      \vdots \\
    \end{array}
  \right)=\tau(A)\left(
    \begin{array}{c}
      \vdots \\
      z^{j0 } \\
      \vdots \\
    \end{array}
  \right).$ Take complex conjugate to get
\begin{equation}\label{2.6}\left(
                             \begin{array}{c}
                               \vdots \\
                                \overline{w^{j0 }}\\
                              \vdots \\
                             \end{array}
                           \right)
=\overline{\tau(A)}\left(
                     \begin{array}{c}
                       \vdots \\
                       \overline{z^{j0 }} \\
                       \vdots \\
                     \end{array}
                   \right).\end{equation}
But we also have $\left(
    \begin{array}{c}
      \vdots \\
      w^{j1 } \\
      \vdots \\
    \end{array}
  \right)=\tau(A)\left(
    \begin{array}{c}
      \vdots \\
      z^{j1 } \\
      \vdots \\
    \end{array}
  \right).$ Note that $w^{(2l)1}=-\overline{w^{(2l+1)0 }}$ and $w^{(2l+1)1}=\overline{w^{(2l)0}}$ by (\ref{2.3}) and similar relation holds for $z^{j\alpha}$. It follows that
\begin{equation}\label{2.5}\left(
                   \begin{array}{c}
                     \vdots \\
                     -\overline{w^{(2l+1)0 } }\\
                     \overline{w^{(2l)0 }} \\
                     \vdots \\
                   \end{array}
                 \right)=\tau(A)\left(
                                         \begin{array}{c}
                                           \vdots \\
                                           -\overline{z^{(2l+1)0 }} \\
                                           \overline{z^{(2l)0 }} \\
                                           \vdots \\
                                         \end{array}
                                       \right).
\end{equation}But \begin{equation*}-J\left(
                   \begin{array}{c}
                     \vdots \\
                     \overline{w^{(2l)0 }} \\
                     \overline{w^{(2l+1)0 } }\\
                     \vdots \\
                   \end{array}
                 \right)=\left(
                   \begin{array}{c}
                     \vdots \\
                     -\overline{w^{(2l+1)0 }} \\
                     \overline{w^{(2l)0 }} \\
                     \vdots \\
                   \end{array}
                 \right),
\end{equation*}(\ref{2.5}) can be written as
\begin{equation}\label{2.7}-J\left(
                               \begin{array}{c}
                                 \vdots \\
                                  \overline{w^{j0 }}\\
                                 \vdots \\
                               \end{array}
                             \right)=\tau(A)(-J)\left(
                                                  \begin{array}{c}
                                                    \vdots \\
                                                    \overline{z^{j0 }} \\
                                                    \vdots \\
                                                  \end{array}
                                                \right).\end{equation}
It follows from (\ref{2.6}) and (\ref{2.7}) that
$J\overline{\tau(A)}=\tau(A)J$. \par

(3) Since $A$ is in
$\text{U}_{\mathbb{H}}(n)$, $\overline{A}^tA=I_n$. By Lemma
\ref{l2.1},
$\tau(\overline{q_1}q_2)=\overline{\tau(q_1)}^t\tau(q_2)$ for
$q_1,q_2\in\mathbb{H}$. Then $\tau\left(\overline{A}^t\right)=\overline{\tau(A)}^t$ and
$I_{2n}=\tau\left(\overline{A}^tA\right)=\tau\left(\overline{A}^t\right)\tau(A)=\overline{\tau(A)}^t\tau(A).
$ It follows that $\tau(A)$ is a complex unitary $(2n\times 2n)$-matrix. This together with (\ref{2.231}) implies $\tau(A)J\tau(A)^t=J$, i.e.,
$\tau(A)$ is symplectic.\end{proof}

 Fix a basis
$\{\omega^0,\omega^1,\ldots,\omega^{2n-1}\}$ of $\mathbb{C}^{2n}$. For $A\in\text{GL}_{\mathbb{H}}(n)$, define the \emph{induced $\mathbb{C}$-linear transformation} of $A$ on $\mathbb{C}^{2n}$ as:
\begin{equation}\label{A.}A.\omega^p=\sum_{j=0}^{2n-1}\AC_{pj} \omega^j,\end{equation}and define the \emph{induced $\mathbb{C}$-linear transformation} of $A$ on $\wedge^{2k}\mathbb{C}^{2n}$ as: $A.(\omega^0\wedge\omega^1\wedge\ldots\wedge\omega^{2k-1})=A.\omega^0\wedge A.\omega^1\wedge\ldots \wedge A.\omega^{2k-1}.$
Therefore for $A\in\text{U}_{\mathbb{H}}(n)$,
\begin{equation}\label{2.35}\begin{aligned}A.\beta_n&=\sum_{i,j}\left[\AC_{0j}\AC_{1i}+\AC_{2j}\AC_{3i}+\ldots+\AC_{(2n-2)j}\AC_{(2n-1)i}\right]\omega^j\wedge \omega^i\\
&=\sum_{j<i}\left[\AC^t_{j0}\AC_{1i}-\AC^t_{j1}\AC_{0i}+\ldots+\AC^t_{j(2n-2)}\AC_{(2n-1)i}\right.\\&\qquad\qquad\qquad\left.-{\AC}^t_{j(2n-1)}\AC_{(2n-2)i}\right]\omega^j\wedge
\omega^i\\&=\sum_{j<i}\left(\AC^tJ\AC\right)_{ji}\omega^j\wedge
\omega^i=\sum_{j<i}J_{ji}\omega^j\wedge \omega^i =\beta_n,
\end{aligned}\end{equation}where the fourth identity follows from Proposition \ref{p2.1}(3) and $\beta_n$ is given by (\ref{bn}). Consequently $A.(\wedge^n
\beta_n)=\wedge^n \beta_n$, i.e.,
\begin{equation}\label{2.38}A.\Omega_{2n}=\Omega_{2n},\end{equation}where $\Omega_{2n}$ is given by (\ref{2.41}). This means that $\Omega_{2n}$ is invariant under
unitary transformations on $\mathbb{H}^n$.\par\vskip4mm

\subsection{The operators $d_0,d_1$ and $\triangle$}
\begin{pro}\label{p1.1}$($1$)$ $d_0d_1=-d_1d_0$.\\
$($2$)$ $d_0^2=d_1^2=0$.\\
$($3$)$ For $F\in C_0^\infty(\Omega,\wedge^{p}\mathbb{C}^{2n})$, $G\in C_0^\infty(\Omega,\wedge^{q}\mathbb{C}^{2n})$, we have\begin{equation*}d_\alpha(F\wedge G)=d_\alpha F\wedge G+(-1)^{p}F\wedge d_\alpha G,\qquad \alpha=0,1.\end{equation*}
\end{pro}
\begin{proof} (1) For any $F=\sum_If_I\omega^I$, we have
\begin{equation*}d_0d_1F=\sum_{i,j,I}\nabla_{i0 }\nabla_{j1 }f_{I}~\omega^i\wedge\omega^j\wedge\omega^I=-\sum_{i,j,I}\nabla_{j1 }\nabla_{i0 }f_{I}~\omega^j\wedge\omega^i\wedge\omega^I=-d_1d_0F,
\end{equation*}by \begin{equation}\label{2.22}\nabla_{i\alpha}\nabla_{j\beta} =\nabla_{j\beta}\nabla_{i\alpha} ,\qquad\text{for}~\alpha,\beta=0,1.\end{equation}${\nabla_{j\alpha}}$'s are mutually commutative since they are scalar differential operators of constant coefficients.\\
(2) \begin{equation*}d_\alpha^2F=\sum_{i,j,I}\nabla_{i\alpha}\nabla_{j\alpha}f_I~\omega^i\wedge\omega^j\wedge
\omega^I=0,\end{equation*} by (\ref{2.22}) and $\omega^i\wedge\omega^j\wedge
\omega^I=-\omega^j\wedge\omega^i\wedge \omega^I.$\\
(3) Write $G=\sum_Jg_J\omega^J$. Since $\omega^k\wedge\omega^I\wedge\omega^J=(-1)^{p}\omega^I\wedge\omega^k\wedge\omega^J$, we have
\begin{equation*}\begin{aligned}d_\alpha(F\wedge G)=&\sum_{k,I,J}\nabla_{k\alpha}(f_Ig_J)~\omega^k\wedge\omega^I\wedge\omega^J\\
=&\sum_{k,I,J}[\nabla_{k\alpha}(f_I)g_J+f_I\nabla_{k\alpha}(g_J)]~\omega^k\wedge\omega^I\wedge\omega^J\\
=&\sum_{k,I }\nabla_{k\alpha}(f_I)~\omega^k\wedge\omega^I\wedge \sum_{ J} g_J\omega^J+(-1)^{p}\sum_{k,I }f_I\omega^I\wedge\sum_{ J}\nabla_{k\alpha }(g_J)\omega^k\wedge\omega^J\\
=&d_\alpha F\wedge G+(-1)^{p}F\wedge d_\alpha G.
\end{aligned}\end{equation*}
\end{proof}
It follows from (\ref{2.228}) and Proposition \ref{p1.1}(1) that \begin{equation}\label{2.235}\triangle
F=\frac{1}{2}\sum_{i,j,I}(\nabla_{i0 }\nabla_{j1 }-\nabla_{i1 }\nabla_{j0 })f_{I}~\omega^i\wedge\omega^j\wedge\omega^I.
\end{equation}
Now for a function $u\in C^2$ we
define\begin{equation}\label{2.10}\triangle_{ij}u:=\frac{1}{2}(\nabla_{i0 }\nabla_{j1 }u-\nabla_{i1 }\nabla_{j0 }u).\end{equation}
$2\triangle_{ij}$ is the determinant of $(2\times2)$-submatrix of $i$-th and $j$-th rows in (\ref{2.4}). Then we can write \begin{equation}\label{2.1000}\triangle u=\sum_{i,j=0}^{2n-1}\triangle_{ij}u~\omega^i\wedge
\omega^j,\end{equation}and for $u_1,\ldots,u_n\in
C^2$,
\begin{equation}\label{2.11}\begin{aligned}\triangle
u_1\wedge\ldots\wedge\triangle
u_n&=\sum_{i_1,j_1,\ldots}\triangle_{i_1j_1}u_1\ldots\triangle_{i_nj_n}u_n~\omega^{i_1}\wedge
\omega^{j_1}\wedge\ldots\wedge \omega^{i_n}\wedge
\omega^{j_n}\\&=\sum_{i_1,j_1,\ldots}\delta^{i_1j_1\ldots
i_nj_n}_{01\ldots(2n-1)}\triangle_{i_1j_1}u_1\ldots\triangle_{i_nj_n}u_n~\Omega_{2n},
\end{aligned}\end{equation} where $\Omega_{2n}$ is given by (\ref{2.41}) and $\delta^{i_1j_1\ldots
i_nj_n}_{01\ldots(2n-1)}$ is given by (\ref{2.111}). Note that $\triangle
u_1\wedge\ldots\wedge\triangle u_n$ is symmetric with respect to the
permutation of $u_1,\ldots,u_n$. In particular, when
$u_1=\ldots=u_n=u$, $\triangle u_1\wedge\ldots\wedge\triangle u_n$
coincides with $(\triangle u)^n:=\wedge^n\triangle u$.

\begin{cor}\label{p2.33}For $u_1,\ldots,
u_n\in C^2$, $\triangle
u_1\wedge\ldots\wedge\triangle u_k$ is closed, $k=1,\ldots,n,$.\end{cor}
\begin{proof}By Proposition \ref{p1.1}(3), we have \begin{equation*}d_\alpha(\triangle u_1\wedge\ldots\wedge\triangle
u_k)=\sum_{j=1}^k\triangle u_1\wedge\ldots\wedge d_\alpha(\triangle
u_j)\wedge\ldots\wedge\triangle u_k,
\end{equation*}for $\alpha=0,1$. Note that \begin{equation*}d_0\triangle=d_0^2d_1=0\qquad \text{and }\qquad d_1\triangle=-d_1^2d_0=0,\end{equation*}by using Proposition \ref{p1.1} (1) (2). It follows that $d_\alpha(\triangle u_1\wedge\ldots\wedge\triangle
u_k)=0$.
\end{proof}
We have the following remarkable identities for $d_0$ and $d_1$, which make $d_0$ and $d_1$ behave as $\partial   $ and $ \overline{\partial}$ in the theory of several complex variables, and simplify our investigation of quaternionic Monge-Amp\`{e}re operator.
\begin{pro}\label{p2.3}For $u_1,\ldots,
u_n\in C^2$, \begin{equation*}\begin{aligned}\triangle u_1\wedge \triangle
u_2\wedge\ldots\wedge\triangle u_n&=d_0(d_1u_1\wedge \triangle
u_2\wedge\ldots\wedge\triangle u_n)=-d_1(d_0u_1\wedge \triangle
u_2\wedge\ldots\wedge\triangle u_n)\\&=d_0d_1(u_1\triangle
u_2\wedge\ldots\wedge\triangle u_n)=\triangle (u_1
\triangle u_2\wedge\ldots\wedge\triangle u_n).
\end{aligned}\end{equation*}
\end{pro}
\begin{proof} It follows from Corollary \ref{p2.33} that\begin{equation*}d_0(\triangle
u_2\wedge\ldots\wedge\triangle u_n)=d_1(\triangle
u_2\wedge\ldots\wedge\triangle u_n)=0.
\end{equation*}By Proposition \ref{p1.1}(3), for $\alpha=0,1$, \begin{equation*}\begin{aligned}d_\alpha(u_1\triangle
u_2\wedge\ldots\wedge\triangle u_n)&=d_\alpha u_1\wedge\triangle
u_2\wedge\ldots\wedge\triangle u_n+u_1d_\alpha(\triangle
u_2\wedge\ldots\wedge\triangle u_n)\\&=d_\alpha u_1\wedge\triangle
u_2\wedge\ldots\wedge\triangle u_n.
\end{aligned}\end{equation*}So we have
\begin{equation*}\begin{aligned}\triangle (u_1
\triangle u_2\wedge\ldots\wedge\triangle u_n)=&d_0d_1(u_1\triangle
u_2\wedge\ldots\wedge\triangle u_n)=d_0(d_1u_1\wedge\triangle
u_2\wedge\ldots\wedge\triangle u_n)\\=&d_0d_1u_1\wedge\triangle
u_2\wedge\ldots\wedge\triangle u_n-d_1u_1\wedge d_0(\triangle
u_2\wedge\ldots\wedge\triangle u_n)\\=&\triangle u_1\wedge \triangle
u_2\wedge\ldots\wedge\triangle u_n.
\end{aligned}\end{equation*}\end{proof}

\subsection{The invariance of $d_\alpha$ and $\triangle$ under quaternionic linear transformations}

Let $A\in\text{GL}_{\mathbb{H}}(n)$, $\widetilde{q}=Aq,~q,\widetilde{q}\in\mathbb{H}^n.$ For
the map $A:q\rightarrow \widetilde{q}=Aq$ and a function $\widetilde{u}(\widetilde{q})$, we define the pulling back function $u(q):=\widetilde{u}(Aq).$ By Proposition \ref{p2.1}(1) we have $\textbf{w}=\AC \textbf{z}$, where $\textbf{z}=\tau(q)=(z^{j\alpha})_{2n\times2}$ and
$\textbf{w}=\tau(\widetilde{q})=(w^{j\alpha})_{2n\times2}$, $j=0,1,\ldots,2n-1$, $\alpha=0,1$. Take a basis
$\{\omega^0,\omega^1,\ldots,\omega^{2n-1}\}$ of $\mathbb{C}^{2n}$. Let $\widetilde{\omega}^p:=A.\omega^p=\sum_{j}\AC_{pj} \omega^j$, where $A.$ is the induced transformation defined in (\ref{A.}). Since $A\in \text{GL}_{\mathbb{H}}(n)$, $\{\widetilde{\omega}^0,\widetilde{\omega}^1,\ldots,\widetilde{\omega}^{2n-1}\}$ is also a basis.

\begin{pro}\label{p2.4}For $u\in C^2$, $d_\alpha u$ is
invariant under quaternionic linear transformations on $\mathbb{H}^n$, i.e., $d_\alpha u(q)=(\widetilde{d_\alpha} \widetilde{u})(Aq)$, $\alpha=0,1$, where $\widetilde{d_\alpha} \widetilde{u}=\sum_j\widetilde{\nabla}_{j\alpha}\widetilde{u}~\widetilde{\omega}^j$, and $\widetilde{\nabla}_{j\alpha}$ is the operator defined by (\ref{2.4}) in terms of the real coordinates $\widetilde{x}$ of $\widetilde{q}$ $(\widetilde{q}_j=\widetilde{x}_{4j}+\textbf{i}\widetilde{x}_{4j+1}+\textbf{j}\widetilde{x}_{4j+2}+\textbf{k}\widetilde{x}_{4j+3})$.
\end{pro}
\begin{proof}We claim the transformation formula of operator $\nabla_{j\alpha}$ under coordinate transformation as
\begin{equation}\label{2.9}\nabla_{j\alpha}u(q)=\sum_{t=0}^{2n-1}\AC_{tj}(\widetilde{\nabla}_{t\alpha}\widetilde{u})(Aq),
\end{equation}where $i,j=0,\ldots,2n-1$, $\alpha,\beta=0,1 $. This identity follows from the chain rule when extended to $\mathbb{C}^{4n}\supseteq \mathbb{R}^{4n}$. First let $u(x)$ be a polynomial on $\mathbb{R}^{4n}$, by (\ref{2.3}) we can extend $u$ to a holomorphic polynomial on $\mathbb{C}^{4n}$: \begin{equation}\label{2.93}\begin{aligned}&U(\ldots,z^{(2l)0},z^{(2l)1},z^{(2l+1)0},z^{(2l+1)1},\ldots)\\
=&u\left(\ldots,\frac{z^{(2l)0}+z^{(2l+1)1}}{2},\frac{z^{(2l+1)1}-z^{(2l)0}}{2i},\frac{z^{(2l+1)0}-z^{(2l)1}}{2},\frac{z^{(2l+1)0}+z^{(2l)1}}{2i},\ldots\right).
\end{aligned}\end{equation} 
So do for $\widetilde{u}(\widetilde{x})$. Since $u(q)=\widetilde{u}(Aq)$ and $\textbf{w}=\AC \textbf{z}$, we have $U(z^{00},z^{01},z^{10},z^{11},\ldots)|_{\tau(\mathbb{H}^n)}=
\widetilde{U}(w^{00},w^{01},w^{10},w^{11},\ldots)|_{\tau(\mathbb{H}^n)}$, from which we find that $U(\mathbf{z} ) =
\widetilde{U}( \AC \textbf{z} ) $ for any $\mathbf{z}\in \mathbb{C}^{4n}$ by $\mathbb{C}\tau(\mathbb{H}^n)=\mathbb{C}^{4n}$. Since $U$ and $\widetilde{U}$ are holomorphic functions on $\mathbb{C}^{4n}$, by the chain rule of the holomorphic variables, we have \begin{equation}\label{2.91}\partial_{z^{j\alpha}}U(z^{00},z^{01},z^{10},z^{11},\ldots)=\sum_{t=0}^{2n-1}\AC_{tj}(\partial_{w^{t\alpha}}\widetilde{U})(w^{00},w^{01},w^{10},w^{11},\ldots),
\end{equation}where $\partial_{z^{j\alpha}}$ and $\partial_{w^{t\alpha}}$ are holomorphic derivatives with respect to holomorphic variables $z^{j\alpha}$ and $w^{t\alpha}$, and $(w^{t\alpha})=\tau(A)(z^{j\alpha})$ for fixed $\alpha$. If we can check that when restricted to $\mathbb{H}^{n}\subseteq\mathbb{C}^{4n}$, \begin{equation}\label{2.92}\begin{aligned}\partial_{z^{j\alpha}}U(z^{00},z^{01},z^{10},z^{11},\ldots)|_{\textbf{z}=\tau(q)}&=\frac{1}{2}\nabla_{j\alpha}u(x),\\
(\partial_{w^{t\alpha}}\widetilde{U})(w^{00},w^{01},w^{10},w^{11},\ldots)|_{\textbf{w}=\tau(\widetilde{q})}&=\frac{1}{2}(\widetilde{\nabla}_{t\alpha}\widetilde{u})(\widetilde{x}),\end{aligned}\end{equation}
then (\ref{2.9}) follows from (\ref{2.91}) by restricting to $\mathbb{H}^{n}$. By definition (\ref{2.93}) and (\ref{2.3}), \begin{equation*}\begin{aligned}&\partial_{z^{(2l)0}}U(z^{00},z^{01},z^{10},z^{11},\ldots)\\
&=\partial_{z^{(2l)0}}\left[u\left(\ldots,\frac{z^{(2l)0}+z^{(2l+1)1}}{2},\frac{z^{(2l+1)1}-z^{(2l)0}}{2i},\frac{z^{(2l+1)0}-z^{(2l)1}}{2},\frac{z^{(2l+1)0}+z^{(2l)1}}{2i},\ldots\right)\right]\\
&=\frac{\partial u(x)}{\partial x_{4l}}\cdot\frac{1}{2}+\frac{\partial u(x)}{\partial x_{4l+1}}\cdot\frac{i}{2}=\frac{1}{2}\nabla_{(2l)0}u(x),\\
&\partial_{z^{(2l)1}}U(z^{00},z^{01},z^{10},z^{11},\ldots)=\frac{\partial u(x)}{\partial x_{4l+2}}\cdot\left(-\frac{1}{2}\right)+\frac{\partial u(x)}{\partial x_{4l+3}}\cdot\left(-\frac{i}{2}\right)=\frac{1}{2}\nabla_{(2l)1}u(x),\\
&\partial_{z^{(2l+1)0}}U(z^{00},z^{01},z^{10},z^{11},\ldots)=\frac{\partial u(x)}{\partial x_{4l+2}}\cdot\frac{1}{2}+\frac{\partial u(x)}{\partial x_{4l+3}}\cdot\left(-\frac{i}{2}\right)=\frac{1}{2}\nabla_{(2l+1)0}u(x),\\
&\partial_{z^{(2l+1)1}}U(z^{00},z^{01},z^{10},z^{11},\ldots)=\frac{\partial u(x)}{\partial x_{4l}}\cdot\frac{1}{2}+\frac{\partial u(x)}{\partial x_{4l+1}}\cdot\left(-\frac{i}{2}\right)=\frac{1}{2}\nabla_{(2l+1)1}u(x).\\
\end{aligned}\end{equation*}By the same reason, the second identity in (\ref{2.92}) holds. Combine (\ref{2.91}) and (\ref{2.92}) to get (\ref{2.9}) when $\widetilde{u},u$ are polynomials. Then (\ref{2.9}) holds for all functions by Taylor's formula. See pp.202 in \cite{kang} and Section 2 of \cite{Wang} for this construction of holomorphic polynomials and the relationship between $\nabla_{j\alpha}$ and $\partial_{z^{j\alpha}}$.

By (\ref{2.9}) we have \begin{equation}\label{2.141}d_\alpha u(q)=\sum_j\nabla_{j\alpha}u(q)~\omega^j=\sum_{j,t}\AC_{tj}(\widetilde{\nabla}_{t\alpha}\widetilde{u})(Aq)~\omega^j
=\sum_t(\widetilde{\nabla}_{t\alpha}\widetilde{u})(Aq)~\widetilde{\omega}^t=(\widetilde{d_\alpha} \widetilde{u})(Aq).
\end{equation}
\end{proof}
\begin{cor}\label{c2.2}$\triangle u$ and $(\triangle u)^n$ are also invariant under quaternionic linear transformations on $\mathbb{H}^n$, i.e., $\triangle u(q)=\widetilde{\triangle} \widetilde{u}(Aq)$.\end{cor}
\proof It follows from (\ref{2.9}) and (\ref{2.141}) that \begin{equation*}\begin{aligned}\triangle u(q)=&d_0d_1u(q)=d_0 \left[(\widetilde{d_1}\widetilde{u})(Aq)\right]=\sum_{it}\nabla_{i0}\left[(\widetilde{\nabla}_{t1}\widetilde{u})(Aq)\right]~\omega^i\wedge\widetilde{\omega}^t\\
=&\sum_{ith}\AC_{hi}\left(\widetilde{\nabla}_{i0}\widetilde{\nabla}_{t1}\widetilde{u}\right)(Aq)~\omega^i\wedge\widetilde{\omega}^t=\sum_{it}\left(\widetilde{\nabla}_{i0}\widetilde{\nabla}_{t1}\widetilde{u}\right)(Aq)~\widetilde{\omega}^i\wedge\widetilde{\omega}^t=(\widetilde{\triangle} \widetilde{u})(Aq).\end{aligned}\end{equation*}\endproof

\section{positive forms and closed positive currents on $\mathbb{H}^n$}\par

\subsection{Positive elements of $\wedge^{2k}\mathbb{C}^{2n}$}
We introduce positive elements of $\wedge^{2k}\mathbb{C}^{2n}$ following Alesker \cite{alesker2}. Since we work on the flat space $\mathbb{H}^n$, they are more concrete. Fix a basis
$\{\omega^0,\omega^1,\ldots,\omega^{2n-1}\}$ of $\mathbb{C}^{2n}$. Under the embedding $\tau$, the right multiplying $\textbf{j}$: $(q_1,\ldots,q_n)\mapsto(q_1\textbf{j},\ldots,q_n\textbf{j})$ maps the left column of (\ref{2.3}) to the right column. It induces a real linear map (up to a sign)
\begin{equation*}\label{rouj}\rho(\textbf{j}):\mathbb{C}^{2n}\rightarrow\mathbb{C}^{2n},\qquad\rho(\textbf{j})(z\omega^k)=\overline{z}J\omega^k,
\end{equation*}which is not $\mathbb{C}$-linear. Also the right multiplying of $\textbf{i}$: $(q_1,\ldots,q_n)\mapsto(q_1\textbf{i},\ldots,q_n\textbf{i})$ induces
\begin{equation*}\label{roui}\rho(\textbf{i}):\mathbb{C}^{2n}\rightarrow\mathbb{C}^{2n},
\qquad\rho(\textbf{i})(z\omega^k)=z\textbf{i}\omega^k.
\end{equation*}Thus $\rho$ defines $GL_{\mathbb{H}}(1)$-action on $\mathbb{C}^{2n}$ ($\rho(\textbf{i})^2=\rho(\textbf{j})^2=-id$, $\rho(\textbf{i}) \rho(\textbf{j})=-\rho(\textbf{j})\rho(\textbf{i})$). $\mathbb{C}^{2n}$ is also a $GL_{\mathbb{H}}(n)$-module by $\tau$, since $\tau(A)$ is   $\mathbb{C}$-linear   on $\mathbb{C}^{2n}$ and $\tau(AB)=\tau(A)\tau(B)$ for   $A,B\in GL_{\mathbb{H}}(n)$. Also \begin{equation}\label{eq:rouj-A}\begin{aligned}\rho(\textbf{j})\circ\tau(A)(z\omega^k)&=\rho(\textbf{j})\left(z\sum_l\tau(A)_{kl}\omega^l\right)=\overline{z}\sum_l\overline{\tau(A)_{kl}}J\omega^l\\
&=\overline{z}J(\overline{\tau(A)}\omega^k)=\overline{z}\tau(A)J \omega^k =\tau(A)\circ\rho(\textbf{j})(z\omega^k)
\end{aligned}\end{equation}
by Proposition \ref{p2.1} (2). It is obvious that $\rho(\textbf{i})\circ\tau(A)=\tau(A)\circ\rho(\textbf{i})$. Thus the actions of $GL_{\mathbb{H}}(1)$ and $GL_{\mathbb{H}}(n)$ on $\mathbb{C}^{2n}$ are commutative, and gives $\mathbb{C}^{2n}$ a structure of $GL_{\mathbb{H}}(n)GL_{\mathbb{H}}(1)$-module. This action extends to $\wedge^{2k}\mathbb{C}^{2n}$ naturally.

An element $\varphi$ of $\wedge^{2k}\mathbb{C}^{2n}$ is called \emph{real} if $\rho(\textbf{j})\varphi=\varphi$. Denote by $\wedge^{2k}_{\mathbb{R}}\mathbb{C}^{2n}$ the subspace of all real elements in $\wedge^{2k}\mathbb{C}^{2n}$. These forms are counterparts of $(k,k)-$forms in complex analysis. In the space
$\wedge^{2k}_{\mathbb{R}}\mathbb{C}^{2n}$ let us define convex cones
$\wedge^{2k}_{\mathbb{R}+}\mathbb{C}^{2n}$ and $\text{SP}^{2k}\mathbb{C}^{2n}$ of
positive and strongly positive elements respectively. The cones are
the same as in \cite{alesker2} when the manifold is flat. The definition in
the case $k=0$ is obvious:
$\wedge^{0}_{\mathbb{R}}\mathbb{C}^{2n}=\mathbb{R}$ and the positive elements
are the usual ones. Consider the case $k=n$. dim$
_{\mathbb{C}}\wedge^{2n}\mathbb{C}^{2n}=1$. One can see that $\Omega_{2n}$ defined by (\ref{2.41}) is an element
of $\wedge_{\mathbb{R}}^{2n}\mathbb{C}^{2n}$ ($\rho(\textbf{j})\beta_n=\beta_n$) and spans it. An
element $\eta\in\wedge_{\mathbb{R}}^{2n}\mathbb{C}^{2n}$ is called\emph{
positive} if $\eta=\kappa~\Omega_{2n}$ for some non-negative number $\kappa$.\par

A right $\mathbb{H}$-linear
map $g:\mathbb{H}^{k}\rightarrow \mathbb{H}^{m}$ induces a $\mathbb{C}$-linear map $\tau(g):\mathbb{C}^{2k}\rightarrow\mathbb{C}^{2m}$. We write $g=(g_{jl})_{m\times k},g_{jl}\in \mathbb{H}$, then $\tau(g)$ is the complex $(2m\times2k)$-matrix$(\tau(g_{jl}))_{j=0,\ldots,m-1}^{l=0,\ldots,k-1}$ given by
(\ref{tau}). Similar to (\ref{A.}), the induced $\mathbb{C}$-linear pulling back transformation of $g^*:\mathbb{C}^{2m}\rightarrow\mathbb{C}^{2k}$ is defined as:\begin{equation}\label{g^*}g^*\widetilde{\omega}^p=\sum_{j=0}^{2k-1}\tau(g)_{pj}\omega^j,\qquad p=0,\ldots,2m-1,\end{equation}where $\{\widetilde{\omega}^0,\ldots,\widetilde{\omega}^{2m-1}\}$ is a basis of $\mathbb{C}^{2m}$ and $\{\omega^0,\ldots,\omega^{2k-1}\}$ is a basis of $\mathbb{C}^{2k}$. The induced $\mathbb{C}$-linear pulling back transformation of $g$ on $\wedge^{2k}\mathbb{C}^{2m}$ is given by: $g^*(\alpha\wedge\beta)=g^*\alpha\wedge g^*\beta $.

An element $\omega\in\wedge_{\mathbb{R}}^{2k}\mathbb{C}^{2n}$ is said to be \emph{elementary strongly positive} if there exist linearly independent right $\mathbb{H}$-linear mappings $\eta_j:\mathbb{H}^n\rightarrow \mathbb{H}$ , $j=1,\ldots,k$, such that \begin{equation}\label{eq:omega}\omega=\eta_1^*\widetilde{\omega}^0\wedge \eta_1^*\widetilde{\omega}^1\wedge\ldots\wedge\eta_k^*\widetilde{\omega}^0\wedge \eta_k^*\widetilde{\omega}^1,\end{equation}where $\{\widetilde{\omega}^0,\widetilde{\omega}^1\}$ is a basis of $\mathbb{C}^{2}$ and $\eta_j^*:~\mathbb{C}^{2}\rightarrow\mathbb{C}^{2n}$ is the induced $\mathbb{C}$-linear pulling back transformation of $\eta_j$. $\omega$ in (\ref{eq:omega}) is real, i.e., $\rho(\textbf{j})\omega=\omega$ (cf. (\ref{eq:rouj-A})).\par

An element $\omega\in\wedge_{\mathbb{R}}^{2k}\mathbb{C}^{2n}$ is called \emph{strongly positive} if it belongs to the convex cone $\text{SP}^{2k}\mathbb{C}^{2n}$ in $\wedge_{\mathbb{R}}^{2k}\mathbb{C}^{2n}$ generated by elementary strongly positive elements; that is, $\omega=\sum_{l=1}^m\lambda_l\xi_l$ for some non-negative numbers $\lambda_1,\ldots,\lambda_m$ and some elementary strongly positive elements $\xi_1,\ldots,\xi_m$. An $2k$-element $\omega$ is said to be \emph{positive} if for any elementary strongly positive element $\eta\in \text{SP}^{2n-2k}\mathbb{C}^{2n}$, $\omega\wedge\eta$ is positive. We will denote the set of all positive $2k$-elements by $\wedge^{2k}_{\mathbb{R}+}\mathbb{C}^{2n}$. The following proposition tells us that any $2k$ element is a $\mathbb{C}$-linear combination of strongly positive $2k$ elements. This fact will be important later.

\begin{pro}\label{p3.12}$($cf. Proposition 5.2 in \cite{alesker2}$)$ \\
$(1)$ $span_\mathbb{C}\{\varphi;~\varphi\in \wedge^{2k}_{\mathbb{R}+}\mathbb{C}^{2n}\}=span_\mathbb{C}\{\varphi;~\varphi\in \text{SP}^{2k}\mathbb{C}^{2n}\}=\wedge^{2k}\mathbb{C}^{2n}$.\\
$(2)~\text{SP}^{2k}\mathbb{C}^{2n}\subseteq\wedge^{2k}_{\mathbb{R}+}\mathbb{C}^{2n}$.\\
$(3)~\text{SP}^{2k}\mathbb{C}^{2n}\wedge\text{SP}^{2l}\mathbb{C}^{2n}\subseteq\text{SP}^{2k+2l}\mathbb{C}^{2n}$.
\end{pro}
\begin{proof}(2) and (3) follow from   definitions directly. Let $V:=span_\mathbb{C}\{\varphi;~\varphi\in \text{SP}^{2k}\mathbb{C}^{2n}\}$. We claim $V$ is an irreducible $GL_{\mathbb{H}}(n) $-module. For any elementary strongly positive $2k$-element $\varphi=\xi_1^*\widetilde{\omega}^0\wedge \xi_1^*\widetilde{\omega}^1\wedge\ldots\wedge\xi_k^*\widetilde{\omega}^0\wedge \xi_k^*\widetilde{\omega}^1$ for some linearly independent right $\mathbb{H}$-linear mappings $\xi_j:\mathbb{H}^n\rightarrow \mathbb{H}$, and $g\in GL_{\mathbb{H}}(n)$, we have \begin{equation}\label{3.21}\begin{aligned}\tau(g)\varphi=&g^*\left(\xi_1^*\widetilde{\omega}^0\wedge \xi_1^*\widetilde{\omega}^1\wedge\ldots\wedge\xi_k^*\widetilde{\omega}^0\wedge \xi_k^*\widetilde{\omega}^1\right)\\=&(\xi_1\circ g)^*\widetilde{\omega}^0\wedge (\xi_1\circ g)^*\widetilde{\omega}^1\wedge\ldots\wedge (\xi_k\circ g)^*\widetilde{\omega}^0\wedge (\xi_k\circ g)^*\widetilde{\omega}^1\in V.
\end{aligned}\end{equation}$\xi_j\circ g:\mathbb{H}^n\rightarrow \mathbb{H}$ is also right $\mathbb{H}$-linear.   Therefore $V$ is $GL_{\mathbb{H}}(n) $-module.

Note that we can write $\varphi=g^*\left(e_1^*\widetilde{\omega}^0\wedge e_1^*\widetilde{\omega}^1\wedge\ldots\wedge e_k^*\widetilde{\omega}^0\wedge e_k^*\widetilde{\omega}^1\right)$ for suitable $g\in GL_{\mathbb{H}}(n)$ since $\xi_1,\ldots,\xi_k$ is linearly independent, where $e_j:\mathbb{H}^n\rightarrow \mathbb{H}$ is the $j$-th projection. Namely $V$ is generated by $e_1^*\widetilde{\omega}^0\wedge e_1^*\widetilde{\omega}^1\wedge\ldots\wedge e_k^*\widetilde{\omega}^0\wedge e_k^*\widetilde{\omega}^1$ under the action of $GL_{\mathbb{H}}(n)$. So it is irreducible.

  $V $  is a complex irreducible $\mathfrak{gl}(n,\mathbb{H})$-module. Obviously, $\mathfrak{gl}(2n,\mathbb{C})=\mathbb{C}\tau(\mathfrak{gl}(n,\mathbb{H}))$ by the embedding $\tau$, and so a complex irreducible $\mathfrak{gl}(n,\mathbb{H})$-module is also a complex irreducible $\mathfrak{gl}(2n,\mathbb{C})$-module.
      Note that $V \subseteq\wedge^{2k}\mathbb{C}^{2n}$ and the latter one is already an irreducible $\mathfrak{gl}(2n,\mathbb{C})$-module. So we must have $V =\wedge^{2k}\mathbb{C}^{2n}$.
\end{proof}

By the following lemma, $z^{k\beta}$'s can be viewed as independent variables formally and $\nabla_{j\alpha}$'s are derivatives with respect to these variables.
\begin{lem}\label{l3.1}$\nabla_{j\alpha}z^{k\beta}=2\delta_j^{k}\delta_{\alpha}^{\beta}$, for $z^{k\beta}$'s given by (\ref{2.3})and $\nabla_{j\alpha}$'s given by  (\ref{2.4}).
\end{lem}
\begin{proof}  Assume that $j=2l,\alpha=0$. By (\ref{2.3}), we have\begin{equation*}\begin{aligned}\nabla_{(2l)0 }z^{(2l)0}= (\partial_{x_{4l}}+\textbf{i}\partial_{x_{4l+1}})(x_{4l}-\textbf{i}x_{4l+1})=2;\\\nabla_{(2l)0}z^{(2l+1)1 }= (\partial_{x_{4l}}+\textbf{i}\partial_{x_{4l+1}})(x_{4l}+\textbf{i}x_{4l+1})=0.\end{aligned}\end{equation*}Note that $\nabla_{(2l)0 }$ is a differential operator with respect to variables $x_{4l}$ and $x_{4l+1}$, while $z^{(2l+1)0 }$ and $z^{(2l)1 }$ are functions of variables $x_{4l+1}$ and $x_{4l+3}$. So $\nabla_{(2l)0 }z^{(2l+1)0 }=\nabla_{(2l)0 }z^{(2l)1 }=0$. And $\nabla_{(2l)0 }z^{k\beta}=0$ for $k\neq 2l,2l+1$. It is similar to check other cases.
\end{proof}

\begin{cor}\label{lem1}$(1)$ $\nabla_{j\alpha}(\|q\|^2)=2\overline{z^{j\alpha}}.$\\
$(2)$ $$\triangle_{(2k)(2k+1)}(\|q\|^2)=-\triangle_{(2k+1)(2k)}(\|q\|^2)=4,$$ and $\triangle_{ij}(\|q\|^2)=0$ for other choices of $i,j$.\end{cor}

\begin{proof}(1) Note that
\begin{equation*}\|q\|^2=\sum_{l=0}^{n-1}(x_{4l}^2+x_{4l+1}^2+x_{4l+2}^2+x_{4l+3}^2)=\sum_{l=0}^{n-1}(z^{(2l)0 }z^{(2l+1)1 }-z^{(2l)1 }z^{(2l+1)0 }).\end{equation*}By using Lemma \ref{l3.1} and definition of $z^{j\alpha}$ in (\ref{2.3}), we have\begin{equation*}\begin{aligned}\nabla_{(2l+1)1 }(\|q\|^2)&=\nabla_{(2l+1)1 }(z^{(2l)0 }z^{(2l+1)1 }-z^{(2l)1 }z^{(2l+1)0 })=2z^{(2l)0 }=2\overline{z^{(2l+1)1}},\\\nabla_{(2l)0}(\|q\|^2)&=\nabla_{(2l)0}(z^{(2l)0 }z^{(2l+1)1 }-z^{(2l)1 }z^{(2l+1)0 })=2z^{(2l+1)1 }=2\overline{z^{(2l)0 }},\\\nabla_{(2l)1 }(\|q\|^2)&=-2z^{(2l+1)0 }=2\overline{z^{(2l)1}},\\
\nabla_{(2l+1)0}(\|q\|^2)&=-2z^{(2l)1 }=2\overline{z^{(2l+1)0 }}.
\end{aligned}\end{equation*}
(2) It follows from (1) and Lemma \ref{l3.1} that $\nabla_{i0 }\nabla_{(2l+1)1 }(\|q\|^2)=\nabla_{i0 }(2z^{(2l)0 })=4\delta^{~~i}_{2l}$. Therefore,
$\triangle_{(2l)(2l+1)}(\|q\|^2)$ $=4,\triangle_{(2l+1)(2l)}(\|q\|^2)=-4,$
and $\triangle_{ij}(\|q\|^2)=0$ otherwise. \end{proof}

It follows that $\frac{1}{8}\triangle(\|q\|^2)=\beta_n$ is a positive
$2$-form and
\begin{equation}\label{3.5}\beta_n^n=\wedge^n\beta_n=n!~\Omega_{2n},\end{equation}is a positive
$2n$-form, where $\Omega_{2n}$ is defined by (\ref{2.41}).\par\vskip4mm

Let $\Omega$ be an open set in $\mathbb{H}^n$. Let $\mathcal{D}_0^{p}(\Omega)$ (respectively, $\mathcal{D}^{p}(\Omega)$) be the space of the $C_0(\Omega)$ (respectively, $C_0^\infty(\Omega)$) functions   valued in $\wedge^{p}\mathbb{C}^{2n}$, i.e.,
\begin{equation}\label{2.21}\mathcal{D}_0^{p}(\Omega)=C_0(\Omega,\wedge^{p}\mathbb{C}^{2n})\quad\text{and}\quad
\mathcal{D}^{p}(\Omega)=C_0^\infty(\Omega,\wedge^{p}\mathbb{C}^{2n}).
\end{equation}The elements of the latter one are often called the \emph{test
$p$-forms}. An element $\eta\in\mathcal{D}^{2k}(\Omega)$ is called a \emph{positive $2k$-form} (respectively, \emph{strongly positive $2k$-form}) if for any $q\in\Omega$, $\eta(q)$ is a positive (respectively, strongly positive) element.

\begin{pro}\label{t3.1}Let $u\in PSH\cap C^2(\Omega)$, then $\triangle u$ is a closed strongly positive
$2$-form.
\end{pro}
\begin{proof} It is closed by Corollary \ref{p2.33}. Note that by (\ref{2.4}) and (\ref{2.10}) we have\begin{equation*}\begin{aligned}\frac{\partial^2u}{\partial \bar{q_l}\partial q_k}&=\left(\frac{\partial }{\partial x_{4l}}+\textbf{i}\frac{\partial }{\partial x_{4l+1}}+\textbf{j}\frac{\partial }{\partial x_{4l+2}}+\textbf{k}\frac{\partial }{\partial x_{4l+3}}\right)\left(\frac{\partial }{\partial x_{4k}}-\textbf{i}\frac{\partial }{\partial x_{4k+1}}-\textbf{j}\frac{\partial }{\partial x_{4k+2}}-\textbf{k}\frac{\partial }{\partial x_{4k+3}}\right)u\\&=\left(\nabla_{(2l)0}+\textbf{j}\nabla_{(2l+1)0}\right)\left(\nabla_{(2k+1)1}-\textbf{j}\nabla_{(2k+1)0}\right)u\\
&=\left(\nabla_{(2l)0}\nabla_{(2k+1)1}-\nabla_{(2l)1}\nabla_{(2k+1)0}\right)u+\textbf{j}\left(\nabla_{(2l+1)0}\nabla_{(2k+1)1}-\nabla_{(2l+1)1}\nabla_{(2k+1)0}\right)u\\
&=2\left(\triangle_{(2l)(2k+1)}u+\textbf{j}\triangle_{(2l+1)(2k+1)}u\right).\end{aligned}\end{equation*}
For any $u\in PSH\cap C^2(\Omega)$, set $\widetilde{u}(\widetilde{q})=u(A\widetilde{q})$. By the Claim (pp.21) in \cite{alesker1} or Corollary 3.1 in
\cite{wang1},\begin{equation}\label{A*A}\left(\frac{\partial^2\widetilde{u}}{\partial
\overline{\widetilde{q}}_l\partial \widetilde{q}_k}(\widetilde{q})\right)=\overline{A}^t\left(\frac{\partial^2u}{\partial
\overline{q}_l\partial
q_k}(A\widetilde{q})\right)A.
\end{equation}By choosing a suitable $A\in \text{GL}_\mathbb{H}(n)$, the right hand side above is a diagonal real matrix. Hence $\widetilde{\triangle}_{(2l)(2k+1)}\widetilde{u}=0$ for  $k\neq l$ and $ \widetilde{\triangle}_{(2l+1)(2k+1)}\widetilde{u}=0$ for all $k, l$. Also by definition,  $\widetilde{\triangle}_{(2l)(2k)}\widetilde{u}=-\overline{\widetilde{\triangle}_{(2l+1)(2k+1)}\widetilde{u}}=0$.
Thus $\widetilde{\triangle}\widetilde{u}(\widetilde{q})=\sum\widetilde{\triangle}_{(2k)(2k+1)}\widetilde{u}(\widetilde{q})~\widetilde{\omega}^{2k}\wedge\widetilde{\omega}^{2k+1},$ which is a strongly positive $2$-form by $\widetilde{\triangle}_{(2k)(2k+1)}\widetilde{u}(\widetilde{q})\geq 0$ (cf. (\ref{eq:A-laplace})). Then the result follows from the invariance of $\triangle u$ in Corollary \ref{c2.2}.\end{proof}
\begin{rem}\label{rem:real}
    Any positive $2k$-form $\varphi$ is automatically real, i.e., $\varphi\in C_0(\Omega,\wedge_{\mathbb{R}}^{2k}\mathbb{C}^{2n})$. By Proposition \ref{p3.12} above, $\varphi$ is a $\mathbb{C}$-linear combination of strongly positive $2k$-elements: $\varphi=\sum a_t\varphi_t+i\sum b_t\varphi_t$, where $a_t,b_t$ are real functions and $\varphi_t$'s are strongly positive. If $\varphi$ is positive, by definition $\sum b_t\varphi_t\wedge\eta=0$ for any strongly positive $(2n-2k)$-element $\eta$, then we get $\sum b_t\varphi_t=0$.
\end{rem}
The following is a criterion of positivity (cf. Proposition 3.2.4 in \cite{klimek} for this criterion in the complex case).
\begin{pro}\label{p3.11}Let $\omega\in\wedge^{2k}\mathbb{C}^{2n}$. Then $\omega\in\wedge^{2k}_{\mathbb{R}+}\mathbb{C}^{2n}$ if and only if for any right $\mathbb{H}$-linear
map $g:\mathbb{H}^{k}\rightarrow \mathbb{H}^{n}$, $g^*\omega\in \wedge^{2k}_{\mathbb{R}+}\mathbb{C}^{2k}$, i.e., $g^*\omega=\kappa~\Omega_{2k}$ for some $\kappa\geq0$.
\end{pro}
\begin{proof} Let $\eta=\eta_{k+1}^*\widetilde{\omega}^0\wedge \eta_{k+1}^*\widetilde{\omega}^1\wedge\ldots\wedge\eta_n^*\widetilde{\omega}^0\wedge \eta_n^*\widetilde{\omega}^1$ be an elementary strongly positive element, for some right $\mathbb{H}$-linearly independent mappings $\eta_j:\mathbb{H}^n\rightarrow \mathbb{H}$ , $j=k+1,\ldots,n$, where $\{\widetilde{\omega}^0,\widetilde{\omega}^1\}$ is a basis of $\mathbb{C}^{2}$. Choose right $\mathbb{H}$-linear mappings $\eta_1,\ldots,\eta_k:\mathbb{H}^n\rightarrow \mathbb{H}$ so that $\eta_1,\ldots,\eta_n$ are right $\mathbb{H}$-linearly independent, and let $v_1,\ldots,v_{n}$ be dual to $\eta_1,\ldots,\eta_n$. By orthogonalization process, we can assume that $v_1,\ldots,v_{n}$ is a unitary orthonormal basis for $\mathbb{H}^{n}$. (If not, let $v_{k+1}'=v_{k+1},v_{k+2}'=v_{k+2}+ v_{k+1}\alpha$ such that $(v_{k+2}',v_{k+1}')=0$ for some proper constant $\alpha\in \mathbb{H}$. And also let $\eta_{k+1}'=\eta_{k+1}-\alpha\eta_{k+2},\eta_{k+2}'=\eta_{k+2}$. Note that $(\eta_{k+1}')^*(\widetilde{\omega}^0\wedge \widetilde{\omega}^1)\wedge(\eta_{k+2}')^*(\widetilde{\omega}^0\wedge \widetilde{\omega}^1)=\eta_{k+1}^*(\widetilde{\omega}^0\wedge \widetilde{\omega}^1)\wedge\eta_{k+2}^*(\widetilde{\omega}^0\wedge \widetilde{\omega}^1)$. Repeating this procedure, $v_{k+1}',\ldots,v_{n}'$ are mutually orthogonal and $\eta$ remains unchanged. Then    normalize them. Now take $v_1',\ldots,v_k'$ such that $v_1',\ldots,v_{n}'$ is a unitary orthonormal basis for $\mathbb{H}^{n}$). See \cite{CMW} for more about quaternionic linear algebra.

 Define $g:\mathbb{H}^{k}\rightarrow \mathbb{H}^{n}$, $g(q_1,\ldots,q_k)=\sum_{l=1}^kv_lq_l$, and $g':\mathbb{H}^{n-k}\rightarrow \mathbb{H}^{n}$, $g'(q_{k+1},\ldots,q_n)=\sum_{l=k+1}^nv_lq_l$.
Since $\eta_j$ is right $\mathbb{H}$-linear, we have
\begin{equation*}\eta_j\left(g'(q_{k+1},\ldots,q_n)\right)=\eta_j\left(\sum_{l=k+1}^nv_lq_l\right)=\sum_{l=k+1}^n\eta_j(v_l)q_l=q_j,
\end{equation*}for $j=k+1,\ldots,n$. Denote $\xi_j=\eta_j\circ g'$, $j=k+1,\ldots,n$. Then $\xi_j:\mathbb{H}^{n-k}\rightarrow \mathbb{H}$ is right $\mathbb{H}$-linear, $\xi_j(q_{k+1},\ldots,q_n)=q_j$, and $\xi_{k+1},\ldots,\xi_n$ are right $\mathbb{H}$-linearly independent. It follows from the definition that $\xi_{k+1}^*\widetilde{\omega}^0\wedge \xi_{k+1}^*\widetilde{\omega}^1\wedge\ldots\wedge\xi_n^*\widetilde{\omega}^0\wedge \xi_n^*\widetilde{\omega}^1$ is an elementary strongly positive element in $\wedge^{2n-2k}\mathbb{C}^{2n-2k}$. Then we have
\begin{equation}\label{sposi}\begin{aligned}&(g,g')^*(\omega\wedge\eta)=g^*\omega\wedge (g')^*\eta=g^*\omega\wedge (g')^*\left[\eta_{k+1}^*\widetilde{\omega}^0\wedge \eta_{k+1}^*\widetilde{\omega}^1\wedge\ldots\wedge \eta_n^*\widetilde{\omega}^0\wedge \eta_n^*\widetilde{\omega}^1\right]\\
&=g^*\omega\wedge (\eta_{k+1}\circ g')^*\widetilde{\omega}^0\wedge (\eta_{k+1}\circ g')^*\widetilde{\omega}^1\wedge\ldots\wedge (\eta_{n}\circ g')^*\widetilde{\omega}^0\wedge (\eta_{n}\circ g')^*\widetilde{\omega}^1\\
&=g^*\omega\wedge \xi_{k+1}^*\widetilde{\omega}^0\wedge \xi_{k+1}^*\widetilde{\omega}^1\wedge\ldots\wedge\xi_n^*\widetilde{\omega}^0\wedge \xi_n^*\widetilde{\omega}^1.
\end{aligned}\end{equation}
Here $(g')^*\eta_{j}^*(\widetilde{\omega}^\alpha)=(\eta_{j}\circ g')^*\widetilde{\omega}^\alpha$ by Proposition \ref{p2.1} (1). If $g^*\omega$ is positive, then $(g,g')^*(\omega\wedge\eta)$ is positive. $(g,g')$ is a unitary transformation which preserves positivity of $2n$-elements by (\ref{2.38}). Hence $\omega\wedge\eta$ is positive. \par
Conversely, let $g:\mathbb{H}^{k}\rightarrow \mathbb{H}^{n}$ be a right $\mathbb{H}$-linear mapping and let $v_1,\ldots,v_{n}$ be a basis for $\mathbb{H}^{n}$ such that $v_1,\ldots,v_{p}$ are the columns of the matrix representing $g$. Let $\eta_1,\ldots,\eta_n$ be a basis dual to $v_1,\ldots,v_{n}$. Define $g'$ as before, then (\ref{sposi}) holds. Hence if $\omega$ is positive, then so is $g^*\omega$.
\end{proof}
\subsection{Positive currents}\par

Now let us discuss positive currents. For more information about the
complex currents see \cite{demailly,harvey2,klimek,lelong2,lelong}.\par  The elements of the dual space $(\mathcal {D}^{2n-p}(\Omega))' $
are called \emph{$p$-currents}. The elements of the dual space
$(\mathcal {D}_0^{2n-p}(\Omega))' $ are called \emph{$p$-currents of
order zero}. Obviously, the $2n$-currents are just the
distributions on $\Omega$, whereas the $2n$-currents of order
zero are Radon measures on $\Omega$.
 Let $\psi$ be a $p$-form whose coefficients are
locally integrable in $\Omega$. One can associate with $\psi$ the $p$-current $T_{\psi}$ defined by
\begin{equation}\label{2.25}T_{\psi}(\varphi)=\int_\Omega\psi\wedge\varphi,
\end{equation}for any $\varphi\in \mathcal {D}^{2n-p}(\Omega)$.\par

Let $I=(i_1,\ldots,i_{2k})$ be multi-index such that $1\leq
i_1<\ldots< i_{2k}\leq n$. Denote by $\widehat{I}=(l_1,\ldots,l_{2n-2k})$ the
\emph{increasing complements} to $I$ in the set $\{0,1,\ldots,2n-1\}$, i.e., $\{i_1,\ldots,i_{2k}\}\cup\{l_1,\ldots,l_{2n-2k}\}=\{1,\ldots,2n\}$. For a $2k$-current $T$ in $\Omega$ and multi-index $I$, define distributions $T_I$ by
$T_I(f)=\varepsilon_IT(f\omega^{\widehat{I}})$ for $f\in
C_0^\infty(\Omega)$, where $\varepsilon_I=\pm1$ is so chosen that
\begin{equation}\label{2.26}\varepsilon_I\omega^I\wedge \omega^{\widehat{I}}=\Omega_{2n}.
\end{equation}
In this way the current $T$ can
be regarded as a $2k$-form with the distributional
coefficients $T_I$, i.e. we can write  
\begin{equation}\label{eq:coef-T}
     T=\sum_IT_I\omega^I,
\end{equation}
  where the
summation is taken over increasing multi-indices of length
$2k$, since $\omega^I$ is dual to $\varepsilon_I\omega^{\widehat{I}}$. If $T$ is a current of order $0$, the
distributions $T_I$ are Radon measures and the current $T$ can be
regarded as a $2k$-form with the measure coefficients, i.e.,
\begin{equation}\label{3.91}T(\varphi)=\sum_I\varepsilon_IT_I(\varphi_{\widehat{I}}),
\end{equation}for $\varphi=\sum_{\widehat{I}}\varphi_{\widehat{I}}\omega^{\widehat{I}}\in \mathcal {D}^{2n-2k}(\Omega)$,
where $I$ and $\widehat{I}$ are increasing. \par
If $T$
is a $2k$-current on $\Omega$, $\psi$ is a $2l$-form on
$\Omega$ with coefficients in $C^{\infty}(\Omega)$, and $k+l\leq n$,
then the
formula\begin{equation}\label{2.27}(T\wedge\psi)(\varphi)=T(\psi\wedge\varphi)
 \qquad\text{for}~~~\varphi\in \mathcal {D}^{2n-2k-2l}(\Omega)
\end{equation}defines a $(2k+2l)$-current. In particular, if $\psi$ is a smooth function, $\psi T(\varphi)=T(\psi\varphi)$.\par

A $2k$-current $T$ is said to be \emph{positive} if we have $T(\eta)\geq0$ for any $\eta\in C_0^\infty(\Omega,SP^{2n-2k}\mathbb{C}^{2n})$. In other words, $T$ is positive if for any $\eta\in C_0^\infty(\Omega,SP^{2n-2k}\mathbb{C}^{2n})$,   $T\wedge\eta=\mu~\Omega_{2n}$ for some    positive
distribution $\mu$ (and hence a measure).
\begin{pro}\label{p3.13}Any positive $2k$-current $T$ on $\Omega$ has measure
coefficients (i.e. is of order zero), and we can write $T=\sum_IT_I\omega^I$ for some complex Radon measures $T_I$, where the
summation is taken over increasing multi-indices $I=(i_1,\ldots,i_{2k})$.\end{pro}
\begin{proof}By Proposition \ref{p3.12}, one can find $\{\varphi_L\}\subseteq SP^{2n-2k}\mathbb{C}^{2n}$ such that any $\eta\in\wedge^{2n-2k}\mathbb{C}^{2n}$ is a $\mathbb{C}$-linear combination of $\varphi_L$, i.e., $\eta=\sum\lambda_L\varphi_L$ for some $\lambda_L\in\mathbb{C}$. Let $\{\widetilde{\varphi_L}\}$ be a basis of $\wedge^{2k}\mathbb{C}^{2n}$ which is dual to $\{\varphi_L\}$. Let $T$ be a positive $2k$-current on $\Omega$. Then $T=\sum T_L\widetilde{\varphi_L}$ with the distributional coefficients $T_L$ by (\ref{eq:coef-T}). If $\psi$ is a nonnegative test function, $\psi\varphi_L\in C_0^\infty(\Omega,SP^{2n-2k}\mathbb{C}^{2n})$. Then $T_L(\psi)=T(\psi\varphi_L)\geq0$ by definition. It follows that $T_L$ is a positive distribution, thus is a positive measure.
\end{proof}
It follows from Proposition \ref{p3.13} that for a positive $2k$-current $T$ and a   test $(2n-2k)$-form $\varphi$, we can write $T\wedge\varphi=\mu~\Omega_{2n}$ for some   Radon measure $\mu$. By (\ref{3.91}), we have
\begin{equation}\label{3.92}T(\varphi)=\int_\Omega T\wedge\varphi.
\end{equation}

The following Proposition is obvious and will be used frequently.
\begin{pro}(1) (linearity) For $2n$-currents $T_1$ and $T_2$ with (Radon) measure coefficients, we have
$$\int_\Omega\alpha T_1+ \beta T_2=\alpha\int_\Omega T_1+ \beta \int_\Omega T_2.$$
(2) If $T_1\leq T_2$ as positive $2n$-currents (i.e. $\mu_1\leq\mu_2$ if we write $T_j=\mu_j\Omega_{2n}$, $j=1,2$), then $\int_\Omega T_1\leq \int_\Omega T_2$.
\end{pro}

\begin{lem}\label{p3.9}$($Stokes-type formula$)$ Assume that $T=\sum_iT_i\omega^{\widehat{i}}$ is a smooth $(2n-1)$-form in $\Omega$, where $\omega^{\widehat{i}}=\omega^0\wedge\ldots\wedge\omega^{i-1}\wedge\omega^{i+1}\wedge\ldots\wedge\omega^{2n-1}$. Then for smooth function $h$, we have \begin{equation}\label{stokes}\int_\Omega hd_\alpha T=-\int_\Omega d_\alpha h\wedge T+\sum_{i=0}^{2n-1}(-1)^{i-1}\int_{\partial\Omega}hT_i~n_{i\alpha}~ dS,\end{equation}where $ n_{i\alpha}$, $i=0,1,\ldots,2n-1$, $\alpha=0,1$, is defined by the matrix:\begin{equation}\label{sss4.3}\begin{aligned}\left(
                                                                                                    \begin{array}{cc}
                                                                                                      n_{(2l)0 }& n_{(2l)1 }\\
                                                                                                      n_{(2l+1)0 }&n_{(2l+1)1 }\\
                                                                                                    \end{array}
                                                                                                  \right):=&\left(
                                                                                                                                  \begin{array}{cc}
                                                                                                                                    n_{4l}+\emph{\textbf{i}}~n_{4l+1}&-n_{4l+2}-\emph{\textbf{i}}~n_{4l+3}\\
                                                                                                                                    n_{4l+2}-\emph{\textbf{i}}~n_{4l+3}&n_{4l}-\emph{\textbf{i}}~n_{4l+1}\\
                                                                                                                                  \end{array}
                                                                                                                                \right),
\end{aligned}\end{equation}$l=0,1,\ldots,n-1$. Here \emph{\textbf{n}}$=(n_0,n_1,\ldots,n_{4n-1})$ is the unit outer normal vector to $\partial\Omega$ and $dS$ denotes the surface measure of $\partial\Omega$. In particular, if $h=0$ on $\partial\Omega$, we have \begin{equation*}\int_\Omega hd_\alpha T=-\int_\Omega d_\alpha h\wedge T,\qquad \alpha=0,1,\end{equation*}
\end{lem}
\begin{proof}Note that \begin{equation*}d_\alpha(hT)=\sum_{k,i} \nabla_{k\alpha }(hT_i)\omega^k\wedge\omega^{\widehat{i}}=\sum_{i} \nabla_{i\alpha }(hT_i)(-1)^{i-1}\Omega_{2n}.\end{equation*} Then
 \begin{equation*}\int_\Omega d_\alpha(hT)=\int_\Omega\sum_{i} \nabla_{i\alpha }(hT_i)(-1)^{i-1}dV=\int_{\partial\Omega}\sum_{i} hT_i(-1)^{i-1}n_{i\alpha}~dS\end{equation*}by (\ref{2.24}) and Stokes' formula. (\ref{stokes}) follows from the above formula and $d_\alpha(hT)=d_\alpha h\wedge T+hd_\alpha T$ by Proposition \ref{p1.1} (3).
\end{proof}

Now let us show that $d_\alpha F$ ($\alpha=0,1$), in the generalized sense (\ref{2.271}), coincides with the original definition given by (\ref{2.228}) when $F$ is smooth. Let $\eta$ be arbitrary $(2n-2k-1)$-test form   compactly supported  in $\Omega$.
It follows from Lemma \ref{p3.9} that $\int_\Omega d_\alpha(F\wedge\eta)=0.$ By Proposition \ref{p1.1} (3), $d_\alpha(F\wedge\eta)=d_\alpha F\wedge\eta+F\wedge d_\alpha\eta$. We have
\begin{equation}\label{3.121}-\int_{\Omega}F\wedge
d_\alpha\eta=\int_{\Omega}d_\alpha F\wedge\eta,\qquad\qquad~i.e.,\qquad(d_\alpha F)(\eta)=-F(d_\alpha\eta).\end{equation}

 And we also define $\triangle F$ in the
generalized sense, i.e., for each test $(2n-2k-2)$-form $\eta$,
\begin{equation}\label{3.8}
(\triangle
F)(\eta):=F(\triangle\eta).\end{equation}

 \begin{pro}\label{p3.10}Let $T$ be a $(2n-1)$-current with $\text{supp}~T\subseteq\Omega$ such that $d_\alpha T$ is a positive $2n$-current. Then \begin{equation*}\int_\Omega d_\alpha T=0.\end{equation*}
\end{pro}
\begin{proof}Let $\chi\geq0$ be a smooth function equal to $1$ on $\Omega$ with compact support in $\Omega'\supseteq\Omega$. Since $d_\alpha T$ is positive, we write $d_\alpha T=\mu\Omega_{2n}$ for some measure $\mu$. Then
\begin{equation*}0\leq\int_\Omega \mu=\int_{\Omega'}\chi ~\mu=(d_\alpha T)(\chi)= -T(d_\alpha\chi).\end{equation*}
Note that $\text{supp}~d_\alpha\chi\subseteq \Omega^c$ and $\text{supp}~T\subseteq\Omega$. So $ T(d_\alpha\chi)=0$. It follows that $\int_\Omega \mu=0.$
\end{proof}

\begin{pro}\label{p3.2}Let $u\in PSH(\Omega)$. Then $\triangle u$ is
a closed positive $2$-current.
\end{pro}

\begin{proof}If $u$ is smooth, $\triangle u$ is a closed positive
$2$-form by   Proposition \ref{t3.1}. When $u$ is not smooth, consider
$u_{\varepsilon}=u*\chi_{\varepsilon}$, where $\chi_{\varepsilon}$
is the standard smoothing kernel. Then $u_{\varepsilon}$ converges
decreasingly to $u$ as $\varepsilon\rightarrow0$ (cf. pp.9 in \cite{alesker4}, as in the complex situation, see pp.63 in \cite{klimek}). It suffices to
show that the coefficients $\triangle_{ij}u_{\varepsilon}\rightarrow
\triangle_{ij}u$ in the sense of weak*-convergence of distributions.
For $\varphi\in C_0^{\infty}(\Omega)$,
\begin{equation*}\begin{aligned}\int\triangle_{ij}u_{\varepsilon}\cdot\varphi
=\int u_{\varepsilon}\cdot\triangle_{ij}\varphi\rightarrow
\int u\cdot\triangle_{ij}\varphi=(\triangle_{ij}u)(\varphi)
\end{aligned}\end{equation*}as $\varepsilon\rightarrow0$, by using the
Lebesgue's dominated convergence theorem. It follows that currents
$\triangle u_{\varepsilon}$ converges to $\triangle u$, and so the
current $\triangle u$ is positive. For any test form $\eta$,
\begin{equation*}(d_\alpha\triangle u)(\eta)=-\triangle
u(d_\alpha\eta)=-\lim_{\varepsilon\rightarrow0}\triangle
u_\varepsilon(d_\alpha\eta)=\lim_{\varepsilon\rightarrow0}(d_\alpha\triangle
u_\varepsilon)(\eta)=0,\end{equation*}$\alpha=0,1,$ where the last identity
follows from Corollary \ref{p2.33}. Note that $u_\varepsilon$ is
smooth, $d_\alpha\triangle u_\varepsilon$ coincides with its usual
definition.\end{proof}

\subsection{Bedford-Taylor  theory in the quaternionic case}\par

We are going to define $\triangle u\wedge
T$ as a closed positive current so that it can be applied in the
context of   non-differentiable $PSH$ functions.
\begin{pro}\label{p3.6}Let $u$ be a smooth $PSH$ function on $\Omega$ and $T$ be a
closed positive $2k$-current. Then \begin{equation*}T(u\triangle
\eta)=T(\eta\wedge\triangle u),\end{equation*}for any   $(2n-2k-2)$-test
form $\eta$. This means that when
$u$ is smooth, $\triangle(uT)$ defined by (\ref{3.8}) coincides with
the product $\triangle u\wedge T$ in its usual sense.
\end{pro}
\begin{proof}By using Proposition \ref{p1.1}, we find that\begin{equation}\label{3.9}\begin{aligned}u\triangle\eta=&ud_0d_1\eta=d_0(ud_1\eta)-d_0u\wedge
d_1\eta=d_0(ud_1\eta)+d_1\eta\wedge d_0u\\=&d_0(ud_1\eta)+d_1(\eta\wedge
d_0u)-\eta\wedge d_1d_0u=d_0(ud_1\eta)+d_1(\eta\wedge
d_0u)+\eta \wedge\triangle u.
\end{aligned}\end{equation}Now let $T$ act on both sides of (\ref{3.9}). Since $T$ is closed, we have $T(d_0\widetilde{\eta})=T(d_1\widetilde{\eta})=0$ for any test $(2n-2k-1)$-form $\tilde{\eta}$ by definition of closedness. Then $T(u\triangle
\eta)=T(\eta\wedge\triangle u)$, and
\begin{equation*}\triangle(uT)(\eta)=uT(\triangle\eta)=T(u\triangle\eta)=T(\eta\wedge\triangle
u)=(\triangle u\wedge T)(\eta).
\end{equation*}The last identity follows from definition (\ref{2.27}).\end{proof}
Now, let $u$ be a locally bounded $PSH$ function and let $T$ be a closed
positive $2k$-current. We can write $T=\sum_IT_I\omega^I$ with measures $T_I$. We define the product of a locally bounded function $u$ and the $2k$-current $T$ to be $$uT:=\sum_IuT_I\omega^I.$$ Here $uT_I$ is also a measure. This definition coincides with (\ref{2.27}) when $u$ is smooth. Since the product of two measure does not make sense in general, we are not able to define
the wedge product of two positive currents in general. As in the case of the
complex Monge-Amp\`{e}re operator (cf. \cite{bed}), motivated by Proposition \ref{p3.6}, we
define\begin{equation}\label{3.110}\triangle u\wedge
T:=\triangle(uT),\end{equation}i.e., $(\triangle u\wedge
T)(\eta):=uT(\triangle\eta)$, where $\eta$ is a   test form.
\par

\begin{pro}\label{p3.7}Let $u$ be a locally bounded $PSH$ function and let $T$ be a closed
positive $2k$-current. Then $\triangle u\wedge T$ defined by (\ref{3.110}) is also a closed positive
current. Moreover, inductively,
\begin{equation}\label{3.111}\triangle
u_1\wedge\ldots\wedge\triangle u_p\wedge T:=\triangle(u_1\triangle
u_2\ldots\wedge\triangle u_p\wedge T)\end{equation} is a closed
positive current, when $u_1,\ldots,u_p\in PSH\cap
L_{loc}^\infty(\Omega)$.
\end{pro}

\begin{proof}  Since $u$ is a locally bounded function and $T$ has
measure coefficients, the current $uT$ is well defined and has
measure coefficients. Therefore $\triangle u\wedge T$ defined by the
formula (\ref{3.110}) is a $(2k+2)$-current and is closed. It suffices to show
the positivity. Let $\eta$ be a strongly positive $(2n-2k-2)$-test form. Let $G\supseteq \text{supp}~\eta$ be a relatively
compact subset of $\Omega$. As before we can use convolution with a family of
regularizing kernels to find a decreasing sequence of smooth
$PSH$ functions $u_m$ converging pointwisely to $u$ as
$m\rightarrow\infty$. Then $u\leq u_m\leq u_1$ and Lebesgue's
dominated convergence theorem shows that $u_{m}T$ converges weakly
to $uT$. $$u_mT=\sum_Iu_mT_I\omega^I\rightarrow\sum_IuT_I\omega^I=uT.$$
Since $u_m$ is smooth and $T$ is closed, $\triangle(u_mT)$
coincides with the product $\triangle u_m\wedge T$ in its usual
sense by Proposition \ref{p3.6}.\par For each $m$, by Proposition
\ref{t3.1}, $\triangle u_{m}$ is a strongly positive $2$-form, and so $\triangle u_{m}\wedge\eta$ is a strongly
positive test $(2n-2k)$-form. Since $T$ is positive, $T(\triangle
u_{m}\wedge\eta)\geq0$. Therefore we have
\begin{equation}\label{3.2}(\triangle u\wedge T)(\eta):=(uT)(\triangle \eta)
=\lim_{m\rightarrow\infty}(u_{m}T)(\triangle
\eta)=\lim_{m\rightarrow\infty}T(u_{m}\triangle\eta)=\lim_{m\rightarrow\infty}T(\triangle u_{m}\wedge\eta)\geq0,
\end{equation}where the last identity follows from Proposition \ref{p3.6}. The proposition is proved.
\end{proof}
In particular, for $u_1,\ldots,u_n\in PSH\cap
L_{loc}^{\infty} (\Omega)$, $\triangle u_1\wedge\ldots\wedge\triangle u_n=\mu\Omega_{2n}$ for a well defined positive Radon
measure $\mu$. For any test $(2n -2p)$-form $\psi$ on $\Omega$, (\ref{3.111}) can be rewritten as
\begin{equation}\label{3.24}\int_\Omega\triangle
u_1\wedge\ldots\wedge\triangle u_p \wedge\psi=\int_\Omega u_1\triangle
u_2\ldots\wedge\triangle u_p \wedge\triangle\psi\end{equation}by (\ref{3.92}), where $u_1,\ldots,u_p\in PSH\cap
L_{loc}^\infty(\Omega)$. Since closed positive currents have measure coefficients,   (\ref{3.24}) also holds for smooth $\psi$   vanishing on the boundary. This fact will be  important later.

Alesker gave a quaternionic version of Chern-Levine-Nirenberg estimate in Proposition 6.3 \cite{alesker2}.   Let $T$ be a positive $(2n-2p)$-current,
$K$ be an arbitrary compact subset. Define
\begin{equation}\label{3.12}\|T\|_K:=\int_K T\wedge\beta_n^p,\end{equation}where
$\beta_n$ is defined by (\ref{bn}). In
particular, if $T$ is a positive $2n$-current,
$\|T\|_K$ coincides with $\int_KT$ defined by
(\ref{2.273}).\par

\begin{pro}\label{p3.88}Let
$\Omega$ be a domain in $\mathbb{H}^n$. Let $K,L$ be compact subsets
of $\Omega$ such that $L$ is contained in the interior of $K$. Then
there exists a constant $C$ depending only on $K,L$ such that for
any $u_1,\ldots u_k\in PSH\cap C^2(\Omega)$, one
has
\begin{equation}\label{3.14}\|\triangle
u_1\wedge\ldots\wedge\triangle u_k \|_{L}\leq
C\prod_{i=1}^k\|u_i\|_{L^{\infty}(K)} ,
\end{equation}where $\|\cdot\|_{L}$ is defined by
$(\ref{3.12})$. \end{pro}

\begin{proof} By Proposition \ref{t3.1}, $\triangle
u_1\wedge\ldots\wedge\triangle u_k $ is already   closed and strongly positive.   Since $L$ is compact, there is a covering of $L$ by a family of balls $B_j'\Subset B_j\subseteq K$. Let $\chi\geq0$ be a smooth function equals to 1 on $\overline{B_j}'$ with support in $B_j$. Then
 \begin{equation*}\begin{aligned}\|\triangle
u_1 \wedge\ldots\wedge\triangle u_k\|_{L\cap\overline{B_j}'}=&
\int_{\overline{B_j}'}\triangle
u_1 \wedge\ldots\wedge\triangle u_k\wedge\beta_n^{n-k}\leq \int_{B_j}\chi\triangle
u_1\wedge \triangle
u_2\wedge\ldots\wedge\triangle u_k\wedge\beta_n^{n-k}\\=&\int_{B_j}u_1\triangle\chi
\wedge\triangle
u_2\wedge\ldots\wedge\triangle u_k\wedge\beta_n^{n-k}\\\leq& \frac 1\varepsilon
 \|u_1\|_{L^{\infty}(K)}\|\triangle\chi \| \int_{B_j}
 \triangle
u_2\wedge\ldots\wedge\triangle u_k\wedge\beta_n^{n-k+1},
\end{aligned}\end{equation*}by using (\ref{3.24}) and the following Lemma \ref{l4.1}. The result follows.
\end{proof}

\begin{lem}\label{l4.1}For  $\eta\in\wedge^{2k}_{\mathbb{R}}\mathbb{C}^{2n}$ with $\|\eta\|\leq 1$, $\beta_n^k\pm\varepsilon\eta$ is a positive $2k$-form for some sufficiently small
$\varepsilon>0$.
\end{lem}
\begin{proof} Any fixed right
$\mathbb{H}$-linear map $\sigma:\mathbb{H}^k\rightarrow\mathbb{H}^n$  can be written as $\sigma=L\circ\pi$, where $L$ is a
unitary isomorphism of $\mathbb{H}^n$ and
$\pi:\mathbb{H}^k\rightarrow\mathbb{H}^n,~\pi(q_1,\ldots,q_k)=(q_1,\ldots,q_k,0,\ldots,0)$.
By (\ref{2.35}), $\beta_n$ is invariant under $L $.
So $\sigma^*\beta_n=\pi^*\beta_n=\beta_k$, thus $\sigma^*\beta_n^k=\beta_k^k=k!~\Omega_{2k}$ is a positive $2k$-form in $\mathbb{H}^k$.
\par Note that $\sigma^*\eta$ is also a real $2k$-form in
$\mathbb{H}^k$, we can write
$\sigma^*\eta=\alpha\Omega_{2k}$, for
some $\alpha\in \mathbb{R}$. Therefore
$\sigma^*(\beta_n^k\pm\varepsilon\eta)=(k!\pm\varepsilon\alpha)\Omega_{2k}$ and the coefficient is positive if $\varepsilon$ is taken
sufficiently small. By Proposition \ref{p3.11}, $\beta_n^k\pm\varepsilon\eta$ is positive.
\end{proof}

Now we are going to show that $\triangle
v_1\wedge\ldots\wedge\triangle v_k $ is continuous in decreasing
sequences when $v_1,\ldots,v_k$ are locally bounded $PSH$ functions. This conclusion in the complex case is
due to Bedford and Taylor  \cite{bedford}.

\begin{thm}\label{t3.2}Let
$v^1,\ldots,v^k\in PSH\cap L_{loc}^{\infty}(\Omega)$. Let
$\{v_j^1\}_{j\in\mathbb{N}},\ldots,\{v_j^k\}_{j\in\mathbb{N}}$ be
decreasing sequences of $PSH$ functions in $\Omega$ such
that $\lim_{j\rightarrow\infty}v_j^t=v^t$   pointwisely in $\Omega$
for each $t $. Then the currents $\triangle
v_j^1\wedge\ldots\wedge\triangle v_j^k $ converge weakly to
$\triangle v^1\wedge\ldots\wedge\triangle v^k $ as
$j\rightarrow\infty$.
\end{thm}
We need the following lemmas to prove the theorem.

\begin{lem}\label{l4.2}(Lemma 1.9 in \cite{demailly}) Let $f_j$ be a decreasing sequence of upper
semi-continuous functions converging to $f$ on some separable
locally compact space $X$ and $\mu_j$ a sequence of positive
measures converging weakly to $\mu$ on $X$. Then every weak limit
$\nu$ of $f_j\mu_j$ satisfies $\nu\leq f\mu$.
\end{lem}
\begin{lem}\label{l4.3}Theorem \ref{t3.2} is true under the additional hypothesis that $\Omega$ is an open ball and the functions $v_j^t$ are smooth in $\Omega$ and coincide outside a compact subset of $\Omega$.
\end{lem}
\begin{proof} We need to show that
for any test $(2n-2k)$-form $\psi$ on $\Omega$,
$$\lim_{j\rightarrow\infty}\int_\Omega\triangle
v_j^1\wedge\ldots\wedge\triangle v_j^k \wedge\psi=\int_\Omega\triangle
v^1\wedge\ldots\wedge\triangle v^k \wedge\psi.$$ Again by induction,
observing that the case $k=1$ is obvious. Assume that the result is true for
$k-1$. According to the inductive definition (\ref{3.111}), it is sufficient
to show
that\begin{equation}\label{3.7}\lim_{j\rightarrow\infty}\int_\Omega v_j^k\triangle
v_j^1\wedge\ldots\wedge\triangle
v_j^{k-1} \wedge\triangle\psi=\int_\Omega v^k\triangle
v^1\wedge\ldots\wedge\triangle v^{k-1} \wedge\triangle\psi
\end{equation}for all real test $(2n-2k)$-forms $\psi$ on $\Omega$ (cf. Remark \ref{rem:real}).\par
Since the results are purely local, it is enough to consider $\psi$   supported   in $B=B(a,r)\subseteq \overline{B}\subseteq\Omega$, where $B$ is such that all $v_j^t$ coincide in a neighborhood of $\partial B$. Furthermore, it is enough
to check (\ref{3.7}) for $\psi$ satisfying following properties:
$\triangle\psi$ is positive in $B$, $\psi$ is smooth in a
neighborhood of $B$, and $\psi=0$ in $\partial B$. Indeed, if
$\widetilde{\psi}$ is a real test $(2n-2k)$-form with compact support contained in $B$, define
$\psi=\rho(\triangle \rho)^{n-k}+\varepsilon\widetilde{\psi}$, where $\varepsilon>0$ and $\rho(q)=\|q-a\|^2-r^2$ for $q\in\mathbb{H}^n$. If
$\varepsilon$ is small enough,
$\triangle\psi=(\triangle\rho)^{n-k+1}+\varepsilon\triangle\widetilde{\psi}$ is
positive by Lemma \ref{l4.1} and the fact that $\triangle \rho=8\beta_n$. This $\psi$ satisfies the above
properties. If (\ref{3.7}) holds for $\psi$ and $\rho(\triangle \rho)^{n-k}$, then it also holds for $\widetilde{\psi}$, as required.

Let $B_1=B(a,r_1)$, $0<r_1<r$ and all the functions $v_j^t$ coincide in $B\backslash \overline{B_1}$. Since $\triangle\psi$ is a positive form   and
$\triangle v_j^1\wedge\ldots\wedge\triangle v_j^{k-1} $ are strongly positive
by Proposition \ref{t3.1}, $\triangle v_j^1\wedge\ldots\wedge\triangle
v_j^{k-1} \wedge\triangle\psi$ are positive by definition. Thus $\triangle v_j^1\wedge\ldots\wedge\triangle
v_j^{k-1} \wedge\triangle\psi=\mu_j\Omega_{2n}$ for some positive measure  $\mu_j$. By the inductive assumption, $\triangle
v_j^1\wedge\ldots\wedge\triangle v_j^{k-1} $ converges to $\triangle
v^1\wedge\ldots\wedge\triangle v^{k-1} $ weakly as
$j\rightarrow\infty$, and so $\mu_j$ converges weakly to $\mu $ as measures with $\triangle
v^1\wedge\ldots\wedge\triangle v^{k-1} \wedge\triangle\psi=\mu \Omega_{2n}$. Then apply Lemma
\ref{l4.2} to $f_j=v_j^k,f=v^k,\mu_j,\mu$ to get\begin{equation*}\lim_{j\rightarrow\infty}\int_B~v_j^k\triangle
v_j^1\wedge\ldots\wedge\triangle v_j^{k-1} \wedge\triangle\psi\leq
\int_B v^k\triangle v^1\wedge\ldots\wedge\triangle
v^{k-1} \wedge\triangle\psi.
\end{equation*}Now it suffices to show that
\begin{equation}\label{3.23}\int_B v^k\triangle v^1\wedge\ldots\wedge\triangle
v^{k-1} \wedge\triangle\psi\leq\liminf_{j\rightarrow\infty}\int_Bv_j^k\triangle
v_j^1\wedge\ldots\wedge\triangle v_j^{k-1} \wedge\triangle\psi.
\end{equation}
Since
$v^k\leq v^k_j$, we have\begin{equation}\label{3.10}\begin{aligned}&\int_B v^k\triangle v^1\wedge\ldots\wedge\triangle
v^{k-1} \wedge\triangle\psi=\int_{\overline{B}_1}+\int_{B\backslash\overline{B}_1}\\\leq &\int_{\overline{B}_1}v^{k}_j\triangle v^1\wedge\ldots\wedge\triangle
v^{k-1} \wedge\triangle\psi+\int_{B\backslash\overline{B}_1}v^k\triangle v^1\wedge\ldots\wedge\triangle
v^{k-1} \wedge\triangle\psi\\
=&\int_{B}v^{k}_j\triangle v^1\wedge\ldots\wedge\triangle
v^{k-1}  \wedge\triangle\psi+\int_{B\backslash\overline{B}_1}(v^k-v^{k}_j)\triangle v^1\wedge\ldots\wedge\triangle
v^{k-1} \wedge\triangle\psi.
\end{aligned}\end{equation}The last integral vanishes, since $v^k=v^k_j$ outside the set $\overline{B}_1$. By using (\ref{3.24}) repeatedly for $\psi$ vanishing on $\partial B$, the first integral is equal to
\begin{equation*}\begin{aligned}&\int_{B}\triangle v^1\wedge\ldots\wedge\triangle
v^{k-1}\wedge\triangle v^{k}_j  \wedge\psi=\int_{B}v^{k-1}\triangle v^1\wedge\ldots\wedge\triangle v^{k-2}
\wedge\triangle v^{k}_j \wedge\triangle\psi\\\leq&\int_{B}v^{k-1}_j\triangle v^1\wedge\ldots\wedge\triangle v^{k-2}
\wedge\triangle v^{k}_j \wedge\triangle\psi=\int_{B}\triangle v^{k-1}_j\wedge\triangle v^1\wedge\ldots\wedge\triangle v^{k-2}
\wedge\triangle v^{k}_j \wedge\psi\\=&\int_{B}v^{k}_j\triangle v^1\wedge\ldots\wedge\triangle v^{k-2}
\wedge\triangle v^{k-1}_j \wedge\triangle \psi.
\end{aligned}\end{equation*}Applying the same argument to $v^{k-2},\ldots,v^1$, we obtain
\begin{equation}\label{3.11}\begin{aligned}&\int_B v^k\triangle v^1\wedge\ldots\wedge\triangle
v^{k-1} \wedge\triangle\psi\leq\int_Bv_j^k\triangle
v_j^1\wedge\ldots\wedge\triangle v_j^{k-1} \wedge\triangle\psi.
\end{aligned}\end{equation}Now let
$j\rightarrow\infty$, we get (\ref{3.23}).
\end{proof}
\begin{lem}\label{l4.4}Theorem \ref{t3.2} is true under the additional hypothesis that $\Omega$ is an open ball and the functions $v_j^t$ coincide outside a compact subset of $\Omega$.
\end{lem}\begin{proof}For $\varepsilon>0$, consider $v_{j\varepsilon}^t:=v_j^t*\chi_{\varepsilon}$, $t=1,\ldots,k$, where $\chi_{\varepsilon}$ is the standard smoothing kernel. As $\varepsilon\rightarrow0$, $v_{j\varepsilon}^t$ converges decreasingly to $v_j^t$. We have already shown that
$$\lim_{\varepsilon\rightarrow0}\int_\Omega\triangle
v_{j\varepsilon}^1\wedge\ldots\wedge\triangle v_{j\varepsilon}^k \wedge\psi=\int_\Omega\triangle
v^1_j\wedge\ldots\wedge\triangle v^k_j \wedge\psi$$ for any test $(2p-2k)$-form $\psi$. By using a diagonalization process, we can find a sequence $\varepsilon_j\rightarrow0$ such that $\widetilde{v}_j^t:=v_{j\varepsilon_j}^t$ decreases to $v^t$ as $j\rightarrow\infty$, and
$$\lim_{j\rightarrow\infty}\int_\Omega\triangle
v_{j}^1\wedge\ldots\wedge\triangle v_{j}^k \wedge\psi=\lim_{j\rightarrow\infty}\int_\Omega\triangle
\widetilde{v}_j^1\wedge\ldots\wedge\triangle \widetilde{v}_j^k \wedge\psi.$$
By Lemma \ref{l4.3}, the last limit is equal to $\int_\Omega\triangle
v^1\wedge\ldots\wedge\triangle v^k \wedge\psi.$
\end{proof}
\emph{Proof of Theorem \ref{t3.2}}.  Choose $B=B(a,r)\subseteq \overline{B}\subseteq\Omega$. Define $\rho(q)=\|q-a\|^2-r^2$ for $q\in\mathbb{H}^n$. Let $0<r_1<r_2<r$. For $t=1,\ldots,k$, as the sequence $\{v_j^t\}$ is decreasing and as $v_j$ is locally bounded,
$\{v_j^t\}$ is locally uniformly bounded. Thus we may suppose that
the ranges of the functions $v_j^t$ are contained in the interval
$[r_1^2-r^2,r_2^2-r^2]$. If we can prove the theorem for the functions $u_j^t=\max\{\rho,v_j^t\}$ and
$u^{t}=\max\{\rho,v^{t}\}$ in $B$, the result will automatically be true for the functions $v_j^t$, $v^t$ in $B(a,r_1)$, because $u_j^t=v_j^t$ and $u^t=v^t$ in $B(a,r_1)$. Since $u_j^t=u^t=\rho$ in $B\backslash B(a,r_2)$, the conclusion follows from Lemma \ref{l4.4} directly.
\qed

\begin{cor}The estimate (\ref{3.14}) remains true for $u_1,\ldots,u_k\in PSH\cap
L_{loc}^{\infty} (\Omega)$.
\end{cor}

\section{The Lelong number of a closed positive current}\par
\subsection{Fundamental solution to the quaternionic Monge-Amp\`{e}re operator}
Let us show that the function $-\frac{1}{\|q\|^2}$ is the fundamental solution the quaternionic Monge-Amp\`{e}re operator.
\begin{lem}\label{lem2}For $q\neq0$,
\begin{equation*}\triangle_{ij}\left(-\frac{1}{\|q\|^2}\right)
=-\frac{4}{\|q\|^6}\left(\overline{M_{ij}}-\sum_{k=0}^{n-1}\delta_{(2k)(2k+1)}^{ij}\|q\|^2\right),
\end{equation*}where $$M_{ij}:=det\left(
                                           \begin{array}{cc}
                                             z^{i0} & z^{i1} \\
                                             z^{j0} & z^{j1} \\
                                           \end{array}
                                         \right)=z^{i0}z^{j1}-z^{i1}z^{j0}.$$\end{lem}
\begin{proof}By Lemma \ref{l3.1} and Corollary \ref{lem1}, we have
\begin{equation*}\nabla_{j\alpha}\left(-\frac{1}{\|q\|^2}\right)=\frac{1}{\|q\|^4}\nabla_{j\alpha}(\|q\|^{2})=\frac{2\overline{z^{j\alpha}}}{\|q\|^4}\end{equation*}
and $\nabla_{(2k+1)0 }(\overline{z^{(2l)1 }})=\nabla_{(2k+1)0 }(-z^{(2l+1)0})=-2\delta^l_k$. So
\begin{equation*}\begin{aligned}&\nabla_{(2k+1)0 }\nabla_{(2l)1 }\left(-\frac{1}{\|q\|^2}\right)=\nabla_{(2k+1)0 }\left(\frac{2}{\|q\|^4}\overline{z^{(2l)1 }}\right)=\frac{-4\delta_k^l}{\|q\|^4}-\frac{8}{\|q\|^6}\overline{z^{(2k+1)0 }}~\overline{z^{(2l)1 }},\\&\nabla_{(2k+1)1 }\nabla_{(2l)0 }\left(-\frac{1}{\|q\|^2}\right)=\nabla_{(2k+1)1 }\left(\frac{2}{\|q\|^4}\overline{z^{(2l)0 }}\right)=\frac{4\delta_k^l}{\|q\|^4}-\frac{8}{\|q\|^6}\overline{z^{(2k+1)1}}~\overline{z^{(2l)0}}.\end{aligned}\end{equation*}
Then by definition (\ref{2.10}) we get \begin{equation*}\triangle_{(2k+1)(2l)}\left(-\frac{1}{\|q\|^2}\right)=\frac{-4\delta_k^l}{\|q\|^4}-\frac{4}{\|q\|^6}\overline{M_{(2k+1)(2l)}}.\end{equation*}Noting that $\triangle_{(2l)(2k+1)}=-\triangle_{(2k+1)(2l)}$, we have \begin{equation*}\triangle_{(2l)(2k+1)}\left(-\frac{1}{\|q\|^2}\right)=\frac{4\delta_k^l}{\|q\|^4}-\frac{4}{\|q\|^6}\overline{M_{(2l)(2k+1)}}.\end{equation*}
And for $k\neq l$, \begin{equation*}\begin{aligned}&\nabla_{(2k)0 }\nabla_{(2l)1 }\left(-\frac{1}{\|q\|^2}\right)=\nabla_{(2k)0 }\left(\frac{2}{\|q\|^4}\overline{z^{(2l)1 }}\right)=-\frac{8}{\|q\|^6}\overline{z^{(2k)0 }}~\overline{z^{(2l)1 }},\\&\nabla_{(2k)1 }\nabla_{(2l)0 }\left(-\frac{1}{\|q\|^2}\right)=\nabla_{(2k)1 }\left(\frac{2}{\|q\|^4}\overline{z^{(2l)0 }}\right)=-\frac{8}{\|q\|^6}\overline{z^{(2k)1}}~\overline{z^{(2l)0}},\end{aligned}\end{equation*}by $\nabla_{(2k)0}\overline{z^{(2l)1 }}=\nabla_{(2k)1 }\overline{z^{(2l)0 }}=0$ in Lemma \ref{l3.1}.
So \begin{equation*}\triangle_{(2k)(2l)}\left(-\frac{1}{\|q\|^2}\right)=\frac{-4}{\|q\|^6}\overline{M_{(2k)(2l)}}.\end{equation*}
Similarly, we can get\begin{equation*}
\triangle_{(2k+1)(2l+1)}\left(-\frac{1}{\|q\|^2}\right)=\frac{-4}{\|q\|^6}\overline{M_{(2k+1)(2l+1)}}.\end{equation*}\end{proof}

\begin{pro}\label{p4.2}  $-\frac{1}{\|q\|^2}$ is a PSH function and $\left(\triangle\left(-\frac{1}{\|q\|^2}\right)\right)^n=\frac{8^nn!\pi^{2n}}{(2n)!}\delta_0$.\end{pro}
\begin{proof}
To show that $-\frac{1}{\|q\|^2}$ is a PSH function, i.e., it is subharmonic on each right quaternionic line, it suffices to show that $u(\lambda)=-\frac{1}{\|\lambda\|^2}$ is a subharmonic function of $\lambda\in\mathbb{H}$. First for $\lambda\neq0$, $u(\lambda)$ is harmonic thus is subharmonic. Then at point $0$, $u(0)=-\infty<L(u,0,1),$ where $L(u,0,1)$ denotes the integral average of $u$ on $\{\|\lambda\|=1\}$. Therefore $-\frac{1}{\|\lambda\|^2}$ is subharmonic and $-\frac{1}{\|q\|^2}$ is a PSH function. Similarly, $-\frac{1}{\|q\|^2+\varepsilon}$ for $\varepsilon>0$ is a PSH function.

We claim that \begin{equation}\label{5.51}\begin{aligned}&\sum_{i_1,j_1,i_2,j_2}\delta^{i_1j_1
i_2j_2\ldots}_{01\ldots(2n-1)}M_{i_1j_1}M_{i_2j_2}\\=&\sum_{i_1,j_1,i_2,j_2}\delta^{i_1j_1
i_2j_2\ldots}_{01\ldots(2n-1)}\left(z^{i_10}z^{j_11}z^{i_20}z^{j_21}-z^{i_10}z^{j_11}z^{i_21}z^{j_20}-z^{i_11}z^{j_10}z^{i_20}z^{j_21}+z^{i_11}z^{j_10}z^{i_21}z^{j_20}\right)=0,\end{aligned}\end{equation}for all other indices fixed. In (\ref{5.51}) we expand all factors $M_{ij}=z^{i0}z^{j1}-z^{i1}z^{j0}$. Note that\begin{equation}\label{5.42}\sum_{i_1,j_1,i_2}\delta^{i_1j_1i_2j_2\ldots}_{012\ldots(2n-1)}z^{i_10 }~z^{j_11 }~z^{i_20 }=0.
\end{equation}This is because, under the permutation of $i_1$ and $i_2$, $z^{i_10 }~z^{j_11 }~z^{i_20 }$ is symmetric while $\delta^{i_1j_1i_2j_2\ldots}_{012\ldots(2n-1)}$ is antisymmetric. And by the same reason,  $\sum_{i_1,j_1,i_2}\delta^{i_1j_1i_2j_2\ldots}_{012\ldots(2n-1)}z^{i_11 }~z^{j_10}~z^{i_20 }=0$. So do the other two sums  in the right hand side of (\ref{5.51}).

By (\ref{2.11}) and Lemma \ref{lem2} for $\frac{-1}{ \|q\|^2+\varepsilon }$, we get
\begin{equation}\label{5.52}\begin{aligned}&\left(\triangle\left(-\frac{1}{\|q\|^2+\varepsilon}\right)\right)^n=\sum_{i_1,j_1,\ldots}\delta^{i_1j_1\ldots
i_nj_n}_{01\ldots(2n-1)}\triangle_{i_1j_1}\left(-\frac{1}{\|q\|^2+\varepsilon}\right)\ldots\triangle_{i_nj_n}\left(-\frac{1}{\|q\|^2+\varepsilon}\right)~\Omega_{2n}\\
&=\left(\frac{-4}{(\|q\|^2+\varepsilon)^3}\right)^n\sum_{i_1,j_1,\ldots}\delta^{i_1j_1\ldots
i_nj_n}_{01\ldots(2n-1)}\left(\overline{M_{i_1j_1}}-\sum_{k_1}\delta^{i_1j_1}_{(2k_1)(2k_1+1)}(\|q\|^2+\varepsilon)\right)\\
&\qquad\qquad\cdots\left(\overline{M_{i_nj_n}}-\sum_{k_n}\delta^{i_nj_n}_{(2k_n)(2k_n+1)}(\|q\|^2+\varepsilon)\right)~\Omega_{2n}\\
&=\left(\frac{-4}{(\|q\|^2+\varepsilon)^3}\right)^n\left[\sum_{k_1,\ldots,k_n}2^n\delta^{(2k_1)(2k_1+1)\ldots
(2k_n)(2k_n+1)}_{01\ldots(2n-1)}(-\|q\|^2-\varepsilon)^n~\Omega_{2n}\right.\\
&\quad+\sum_{k_1,\ldots,k_n}2^n\delta^{(2k_1)(2k_1+1)\ldots
(2k_n)(2k_n+1)}_{01\ldots(2n-1)}\left(\sum_{s=1}^n\overline{M_{(2k_s)(2k_s+1)}}\right)(-\|q\|^2-\varepsilon)^{n-1}~\Omega_{2n}\\
&\left.\quad+\ldots+\sum_{i_1,j_1,\ldots}\delta^{i_1j_1\ldots
i_nj_n}_{01\ldots(2n-1)}\overline{M_{i_1j_1}}\cdots\overline{M_{i_nj_n}}~\Omega_{2n}\right].
\end{aligned}\end{equation}Note that in the right hand side above, except for the first two sums, all other sums vanish by (\ref{5.51}). By straightforward computation, \begin{equation*}\begin{aligned}&\sum_{k_1,\ldots,k_n}2^n\delta^{(2k_1)(2k_1+1)\ldots
(2k_n)(2k_n+1)}_{01\ldots(2n-1)}(-\|q\|^2-\varepsilon)^n=2^nn!(-\|q\|^2-\varepsilon)^n,\\
&\sum_{k_1,\ldots,k_n}2^n\delta^{(2k_1)(2k_1+1)\ldots
(2k_n)(2k_n+1)}_{01\ldots(2n-1)}\left(\sum_{s=1}^n\overline{M_{(2k_s)(2k_s+1)}}\right)(-\|q\|^2-\varepsilon)^{n-1}=2^nn!\|q\|^2(-\|q\|^2-\varepsilon)^{n-1},
\end{aligned}\end{equation*}by the fact that $\|q\|^2=\sum_{l=0}^nM_{(2l)(2l+1)}$. It follows that the right hand side of (\ref{5.52}) equals to
\begin{equation*} \frac{8^n n!\varepsilon}{(\|q\|^2+\varepsilon)^{2n+1}} .\end{equation*}

Let $\varepsilon=0$, we get \begin{equation*}
\left(\triangle\left(-\frac{1}{\|q\|^2}\right)\right)^n=0,\end{equation*}for $q\neq0$. Similar to the case of complex Monge-Amp\`{e}re operator (cf. Proposition 6.3.2 in \cite{klimek}), the weak convergence of $(\triangle u)^n$ can be extended slightly to the functions with one pole, i.e., for $u\in PSH(\Omega)\cap L_{loc}^{\infty}(\Omega\backslash a)$ with some point $a\in \Omega$, the weak convergence also holds. So
\begin{equation}\label{b2}\begin{aligned}\int_{\|q\|=1}\left(\triangle\left(-\frac{1}{\|q\|^2}\right)\right)^ndV
=&\lim_{\varepsilon\rightarrow0}\int_{\|q\|=1}\left(\triangle\left(-\frac{1}{\|q\|^2+\varepsilon}\right)\right)^ndV \\
=&\lim_{\varepsilon\rightarrow0}S_{4n}\int_0^1\frac{8^nn!\varepsilon r^{4n-1}}{(\varepsilon+r^2)^{2n+1}}dr=S_{4n}\frac{8^nn!}{2}\int_0^1\frac{\varepsilon t^{2n-1}}{(\varepsilon+t)^{2n+1}}dt\\
=&\lim_{\varepsilon\rightarrow0}S_{4n}\frac{8^nn!}{2}\int_\varepsilon^{\varepsilon+1}\frac{\varepsilon\sum_{k=0}^{2n-1}C^k_{2n-1}t^k(-\varepsilon)^{2n-1-k}}{t^{2n+1}}dt\\
=&\lim_{\varepsilon\rightarrow0}S_{4n}\frac{8^nn!}{2}\sum_{k=0}^{2n-1}\varepsilon\frac{(-\varepsilon)^{2n-1-k}}{k-2n}C^k_{2n-1}\left[(1+\varepsilon)^{k-2n}-\varepsilon^{k-2n}\right]\\
=&S_{4n}\frac{8^nn!}{2}\sum_{k=0}^{2n-1}\frac{(-1)^{2n-k}}{k-2n}C^k_{2n-1},
\end{aligned}
\end{equation}
where $S_{4n}=4n\frac{\pi^{2n}}{(2n)!}$.
Denote $F(x)=\sum_{k=0}^{2n-1}\frac{(-1)^{2n-k}}{k-2n}C^k_{2n-1}x^{2n-k}$. We have $F(0)=0$ and
\begin{equation*}F'(x)=\sum_{k=0}^{2n-1}(-1)^{2n-1-k}C^k_{2n-1}x^{2n-k-1}=(1-x)^{2n-1}.
\end{equation*}So $F(1)=F(0)+\int_0^1(1-x)^{2n-1}dx=\frac{1}{2n}$. Therefore the right hand side of (\ref{b2}) equals to $S_{4n}\frac{8^nn!}{2}F(1)=\frac{8^nn!\pi^{2n}}{(2n)!}.$ The proposition is proved.
\end{proof}

\subsection{The Lelong number}

For a $(2n-2p)$-current $T=\sum_IT_I~\omega^I$, define \begin{equation}\label{t*}T_\varepsilon:=T*\chi_\varepsilon=\sum_I(T_I*\chi_\varepsilon)~\omega^I,\end{equation} where $\chi_\varepsilon$ is the smoothing kernel and $T_I$'s are distributions. For any test form $\eta\in\mathcal {D}^{2p}(\Omega)$, $T(\eta)=\sum_I\varepsilon_IT_I(\eta_{\widehat{I}})$ by (\ref{3.91}) and $$T_I(\eta_{\widehat{I}})=\lim_{\varepsilon\rightarrow0}T_I(\chi_\varepsilon*\eta_{\widehat{I}})=\lim_{\varepsilon\rightarrow0}(T_I*\chi_\varepsilon)(\eta_{\widehat{I}}).$$ Hence $T=\lim_{\varepsilon\rightarrow0}T_\varepsilon$.

 \begin{lem}\label{l5.2}For closed positive current $T$, $T_\varepsilon$ given by (\ref{t*}) is also closed and positive.\end{lem}
 \proof First, for any test form $\eta\in\mathcal {D}^{2p-1}(\Omega)$, $\alpha=0,1$, by (\ref{2.271}), $$(d_\alpha T_\varepsilon)(\eta)=-T_\varepsilon(d_\alpha\eta)=-T(\chi_\varepsilon*d_\alpha\eta)=-T(d_\alpha(\chi_\varepsilon*\eta))=d_\alpha T(\chi_\varepsilon*\eta)=0,$$
where the last identity follows from the fact that $T$ is closed. So $T_\varepsilon$ is also closed. To show the positivity of $T_\varepsilon$, by definition it suffices to prove that $\chi_\varepsilon*\eta$ is in $C_0^\infty(\Omega ,SP^{2p}\mathbb{C}^{2n})$ for each test form $\eta\in C_0^\infty(\Omega,SP^{2p}\mathbb{C}^{2n})$ and small $\varepsilon>0$. For each $\eta\in C_0^\infty(\Omega,SP^{2p}\mathbb{C}^{2n})$, $\eta(p)$ is a strongly positive element, i.e., $\eta(p)=\sum_l\lambda_l\xi_l$ for some $\lambda_l\geq0$ and elementary strongly positive elements $\xi_l$. Since $(\chi_\varepsilon*\eta)(q)=\sum_l(\chi_\varepsilon*\lambda_l)(q)\xi_l$ and $(\chi_\varepsilon*\lambda_l)(q)$ is also nonnegative, $\chi_\varepsilon*\eta$ is a strongly positive form by definition. Therefore $T_\varepsilon(\eta)=T(\chi_\varepsilon*\eta)\geq0$ by the positivity of $T$, and so $T_\varepsilon$ is positive.\endproof

\begin{pro}\label{p5.1}Suppose that $\Omega$
is a domain, $a\in\Omega$, $B(a,R)\Subset\Omega$, and $T$ is a closed
positive $(2n-2p)$-current. Then for $0<r_1<r_2<R$,\begin{equation}\label{5.15}\int_{B(a,r_2)\backslash \overline{B}(a,r_1)}T\wedge\left(\triangle
\left(-\frac{1}{\|q\|^{2}}\right)\right)^p=\frac{\sigma_T(a,r_2)}{r_2^{4p}}-\frac{\sigma_T(a,r_1)}{r_1^{4p}},
\end{equation}where $\sigma_T(a,r)$ is defined by (\ref{5.7}). Thus
$r^{-4p}\sigma_{T}(a,r)$ is an increasing function of $r$. And the number
$\nu_a(T)$ defined
by (\ref{5.8}) exists and is nonnegative.
\end{pro}
\begin{proof}Without loss of generality, we may assume that $a=0$. We first assume that $T$ is smooth. Note that the unit outer normal vector to the sphere $\partial B(0,r)$ is $\textbf{n}=(\frac{x_0}{r},\ldots,\frac{x_j}{r},\ldots\frac{x_{4n-1}}{r})$. By (\ref{2.3}), we have
\begin{equation}\label{5.1}\left(
                             \begin{array}{cc}
                               \overline{z^{(2l)0 }} & \overline{z^{(2l)1 }} \\
                               \overline{z^{(2l+1)0 }} & \overline{z^{(2l+1)1 }} \\
                             \end{array}
                           \right)=\left(
                                      \begin{array}{cc}
                                        x_{4l}+\textbf{i}x_{4l+1} & -x_{4l+2}-\textbf{i}x_{4l+3} \\
                                        x_{4l+2}-\textbf{i}x_{4l+3} & x_{4l}-\textbf{i}x_{4l+1} \\
                                      \end{array}
                                    \right).
\end{equation}By definition of $n_{j0}$ in (\ref{sss4.3}) and (\ref{5.1}), we get for each $j$, \begin{equation}\label{5.3}n_{j0}=\frac{1}{r}\overline{z^{j0 }}\quad\qquad \text{on}\quad\partial B(0,r).\end{equation}\par
Assume that $T$ is of the form $\sum_LT_L\omega^L$ with $T_L$ smooth, where the
multi-index $L=(l_1,\ldots,l_{2n-2p})$ and
$\omega^L:=\omega^{l_1}\wedge\ldots\wedge\omega^{l_{2n-2p}}$. Then
\begin{equation}\label{5.111}d_0[d_1u\wedge(\triangle u)^{p-1}\wedge T]
=\sum_{i_1}\nabla_{i_10 }\left[\sum_{j_1,i_2,\ldots}\nabla_{j_11 }u\triangle_{i_2j_2}u\ldots\triangle_{i_pj_p}u~T_L~\delta^{i_1j_1\ldots
i_pj_pL}_{01\ldots(2n-1)}\right]~\Omega_{2n}.
\end{equation}
Note that $d_0d_1u\wedge(\triangle u)^{p-1}\wedge T=d_0[d_1u\wedge(\triangle u)^{p-1}\wedge T]$ by the fact that $T$ and $\triangle u$ are both closed. Then by (\ref{2.24}) and (\ref{5.111}), we have \begin{equation}\label{5.2}\begin{aligned}&\int_{B(0,r_2)\backslash \overline{B}(0,r_1)}T\wedge(\triangle
u)^p=\int_{B(0,r_2)\backslash \overline{B}(0,r_1)}d_0d_1u\wedge(\triangle u)^{p-1}\wedge T\\
&=\int_{B(0,r_2)\backslash \overline{B}(0,r_1)}d_0[d_1u\wedge(\triangle u)^{p-1}\wedge T]\\
&=\int_{B(0,r_2)\backslash \overline{B}(0,r_1)}\sum_{i_1}\nabla_{i_10 }\left[\sum_{j_1,i_2,\ldots}
\nabla_{j_11 }u\triangle_{i_2j_2}u\ldots\triangle_{i_pj_p}u~T_L~\delta^{i_1j_1\ldots
i_pj_pL}_{01\ldots(2n-1)}\right]dV\\
&=\left(\int_{\partial B(0,r_2)}-\int_{\partial B(0,r_1)}\right)\sum_{i_1,j_1,\ldots}
\nabla_{j_11 }u\triangle_{i_2j_2}u\ldots\triangle_{i_pj_p}u~T_L~\delta^{i_1j_1\ldots
i_pj_pL}_{01\ldots(2n-1)}~n_{i_10}~dS\\
&=\int_{\partial B(0,r_2)}\sum_{i_1,j_1,\ldots}\nabla_{j_11 }u\triangle_{i_2j_2}u\ldots\triangle_{i_pj_p}u~T_L~\delta^{i_1j_1\ldots
i_pj_pL}_{01\ldots(2n-1)}\frac{1}{r_2}\overline{z^{i_10 }}dS\\
&\qquad-\int_{\partial B(0,r_1)}\sum_{i_1,j_1,\ldots}\nabla_{j_11 }u\triangle_{i_2j_2}u\ldots\triangle_{i_pj_p}u~T_L~\delta^{i_1j_1\ldots
i_pj_pL}_{01\ldots(2n-1)}\frac{1}{r_1}\overline{z^{i_10 }}dS,
\end{aligned}\end{equation}where the last
identity follows from (\ref{5.3}).\par

Note that\begin{equation}\label{5.4}\sum_{i_1,i_s,j_s}\delta^{i_1j_1\ldots
i_sj_s\ldots
i_pj_pL}_{012\ldots(2n-1)}\overline{M_{i_sj_s}}~\overline{z^{i_10 }}=0,
\end{equation}by the same reason as (\ref{5.42}). By Lemma \ref{lem2}, \begin{equation}\label{5.41}\begin{aligned}&\sum_{i_1,i_s,j_s}\delta^{i_1j_1\ldots
i_sj_s\ldots
i_pj_pL}_{012\ldots(2n-1)}\triangle_{i_sj_s}\left(-\frac{1}{\|q\|^{2}}\right)\overline{z^{i_10 }}=\sum_{i_1,k_s}\delta^{i_1j_1\ldots
(2k_s)(2k_s+1)\ldots
i_pj_pL}_{012\ldots(2n-1)}\frac{4}{\|q\|^4}\overline{z^{i_10 }}\\
+&\sum_{i_1,k_s}\delta^{i_1j_1\ldots
(2k_s+1)(2k_s)\ldots
i_pj_pL}_{012\ldots(2n-1)}\left(-\frac{4}{\|q\|^4}\right)\overline{z^{i_10 }}+\sum_{i_1,i_s,j_s}\delta^{i_1j_1\ldots
i_sj_s\ldots
i_pj_pL}_{012\ldots(2n-1)}\left(-\frac{4}{\|q\|^6}\overline{M_{i_sj_s}}\right)\overline{z^{i_10 }}.
\end{aligned}\end{equation}It follows from (\ref{5.4}) that the last item in (\ref{5.41}) vanishes. By (\ref{5.4}) and the fact that
$$\delta^{i_1j_1\ldots
(2k_s)(2k_s+1)\ldots
i_pj_pL}_{012\ldots(2n-1)}=-\delta^{i_1j_1\ldots
(2k_s+1)(2k_s)\ldots
i_pj_pL}_{012\ldots(2n-1)},$$ the right hand side of (\ref{5.41}) equals to $$\sum_{i_1,k_s}\delta^{i_1j_1\ldots
(2k_s)(2k_s+1)\ldots
i_pj_pL}_{012\ldots(2n-1)}\left(\frac{8}{\|q\|^4}\right)\overline{z^{i_10 }}.$$
Repeating this process to get
\begin{equation}\label{5.5}\begin{aligned}&\int_{\partial B(0,r)}\sum_{i_1,j_1,\ldots}\nabla_{j_11 }u\triangle_{i_2j_2}u\ldots\triangle_{i_pj_p}u~T_L~\delta^{i_1j_1\ldots
i_pj_pL}_{01\ldots(2n-1)}\frac{1}{r}\overline{z^{i_10 }}dS\\=&\int_{\partial B(0,r)}\sum_{\substack{L,i_1,j_1,
\\k_2\ldots,k_p}}\delta^{i_1j_1(2k_2)(2k_2+1)\ldots
(2k_p)(2k_p+1)L}_{012\ldots(2n-1)}
\nabla_{j_11 }(-\|q\|^{-2})~\left(\frac{8}{r^4}\right)^{p-1}T_L~\frac{1}{r}\overline{z^{i_10 }}dS,\end{aligned}\end{equation}for each $r$ and $u=-\|q\|^{-2}$.

On the other hand, by Corollary \ref{lem1}, apply (\ref{5.2}) to $u=\|q\|^{2}$ to get
\begin{equation}\label{5.6}\begin{aligned}\sigma_T(0,r)=\int_{\partial B(0,r)}\sum_{\substack{L,i_1,j_1,
\\k_2\ldots,k_p}}\delta^{i_1j_1(2k_2)(2k_2+1)\ldots
(2k_p)(2k_p+1)L}_{012\ldots(2n-1)}
\nabla_{j_11 }(\|q\|^{2})~8^{p-1}T_L~\frac{1}{r}\overline{z^{i_10 }}dS.\end{aligned}\end{equation}
Note that
$\nabla_{j_11 }(-\|q\|^{-2})=\|q\|^{-4}\nabla_{j_11 }(\|q\|^{2})$
for each $j_1$.  Compare (\ref{5.5}) with
(\ref{5.6}) to get \begin{equation*}\int_{\partial B(0,r)}\sum_{i_1,j_1,\ldots}\nabla_{j_11 }u\triangle_{i_2j_2}u\ldots\triangle_{i_pj_p}u~T_L~\delta^{i_1j_1\ldots
i_pj_pL}_{01\ldots(2n-1)}\frac{1}{r}\overline{z^{i_10 }}dS=\frac{\sigma_T(0,r)}{r^{4p}},
\end{equation*}for $u=-\|q\|^{-2}$.
Then by (\ref{5.2}) we get (\ref{5.15}). \par
For the case of $T_L$ nonsmooth, we consider $T_\varepsilon=T*\chi_\varepsilon$. By Lemma \ref{l5.2}, we can apply (\ref{5.15}) to $T_\varepsilon$ and let $\varepsilon$ go to zero. Then (\ref{5.15}) holds for $T$. Note that $\triangle
(-\|q\|^{-2})$ is strongly positive on $B(a,r_2)\backslash \overline{B}(a,r_1)$ since $-\|q\|^{-2}$ is  PSH. It follows from $T\wedge(\triangle
(-\|q\|^{-2}))^p\geq0$ that $r^{-4p}\sigma_T(a,r)$ is an increasing
function of $r$.
\end{proof}

\section{Lelong-Jensen type formula}\par
We denote by $\triangle_n(u_1,\ldots, u_n)$ the
coefficient of the form $\triangle u_1\wedge\ldots\wedge\triangle
u_n$, i.e., $\triangle u_1\wedge\ldots\wedge\triangle
u_n=\triangle_n(u_1,\ldots, u_n)~\Omega_{2n}.$ Then we have
\begin{equation}\label{2.15}\triangle_n(u_1,\ldots,
u_n)=\sum_{i_1,j_1,\ldots }\delta^{i_1j_1\ldots
i_nj_n}_{01\ldots(2n-1)}\triangle_{i_1j_1}u_1\ldots\triangle_{i_nj_n}u_n.\end{equation}
When $u_1=\ldots=
u_n=u$,\begin{equation}\label{2.23}\triangle_nu:=\triangle_n(u,\ldots,
u)=\sum_{i_1,j_1,\ldots }\delta^{i_1j_1\ldots
i_nj_n}_{01\ldots(2n-1)}\triangle_{i_1j_1}u\ldots\triangle_{i_nj_n}u.\end{equation}
We give the following explicit representation for the quaternionic boundary measure $\mu_{\varphi,r}$ defined by (\ref{4.1}).
\begin{pro}\label{p4.1}Let $\Omega$
be a quaternionic strictly pseudoconvex domain. Let $\varphi$ be a continuous
$PSH$ function on $\Omega$. Suppose
that $\varphi$ is smooth near $S_\varphi(r)$ and $d\varphi\neq0$ on
$S_\varphi(r)$, then the quaternionic boundary measure
$\mu_{\varphi,r}$ is given by
\begin{equation}\label{4.2}\mu_{\varphi,r}=\sum_{i_1,j_1,\ldots}\delta^{i_1j_1\ldots
i_nj_n}_{01\ldots(2n-1)}\nabla_{j_11}\varphi\triangle_{i_2j_2}\varphi\ldots\triangle_{i_nj_n}\varphi
\cdot n_{i_10}~dS,
\end{equation}where $ n_{i_10}$ is given by (\ref{sss4.3}) and $dS$ denotes the surface measure of $S_\varphi(r)$.
\end{pro}
\begin{proof} Let $h$ be a smooth function with compact support near
$S_\varphi(r)$. Consider the smooth approximation of the cut-off function $\varphi_r$. let $\chi$ be a decreasing sequence of smooth
convex functions on $\mathbb{R}^1$ satisfying
\begin{equation*}\chi_{l}(t)=\begin{cases}
 \,\,\, r,\,\,\,\,\,\,\,\,\,\,\,\,\,\text{if}\,\,\,\,t\leqq r-\frac{1}{l},\\
 \,\,\,t,\,\,\,\,\,\,\,\,\,\,\,\,\,\,\text{if}\,\,\,\,t\geq r+\frac{1}{l},\\
 \end{cases}
 \end{equation*}and $0\leq\chi'_{l}\leq 1$. Then
 \begin{equation*}
 \lim_{l\rightarrow
 +\infty}\chi'_{l}(t)=\begin{cases}
 0 ,& \,\,\text{if}\,\,\,\,t< r, \\
 1 ,& \,\,\text{if}\,\,\,\,t> r,
 \end{cases}
 \end{equation*} and we can write
 $ \varphi_{r}=\lim_{l\rightarrow +\infty}\chi_{l}(\varphi)$. Since
$\varphi$ is smooth near $S_\varphi(r)$,
$(\triangle[\chi_{l}(\varphi)])^n$ tends to $(\triangle\varphi_{r})^n$
weakly as $l\rightarrow +\infty$. By using Stokes-type formula in Lemma \ref{p3.9} repeatedly, we find that
\begin{equation}\label{4.4}\begin{aligned}\int_{\Omega}h(\triangle\varphi_{r})^n=&\lim_{l\rightarrow
+\infty}\int_{\Omega}h(\triangle[\chi_{l}(\varphi)])^n
=\lim_{l\rightarrow +\infty}\int_{\Omega}hd_0\left[d_1\left[\chi_{l}(\varphi)\right]\wedge(\triangle[\chi_{l}(\varphi)])^{n-1}\right]\\
=&\lim_{l\rightarrow+\infty}\int_{\Omega}-d_0h\wedge d_1\left[\chi_{l}(\varphi)\right]\wedge(d_0d_1[\chi_{l}(\varphi)])^{n-1}\\
=&\lim_{l\rightarrow+\infty}\int_{\Omega}-\chi_{l}'(\varphi)d_0h\wedge d_1\varphi\wedge\left[\chi_{l}''(\varphi)d_0\varphi\wedge d_1\varphi+\chi_{l}'(\varphi)d_0d_1\varphi\right]^{n-1}\\
=&\lim_{l\rightarrow+\infty}\int_{\Omega}-(\chi_{l}'(\varphi))^nd_0h\wedge d_1\varphi\wedge(d_0d_1\varphi)^{n-1}\\
=&-\int_{\Omega\backslash
B_\varphi(r)}d_0h\wedge d_1\varphi\wedge(d_0d_1\varphi)^{n-1}\\
=&\int_{\Omega\backslash
B_\varphi(r)}h(d_0d_1\varphi)^n-\int_{\Omega\backslash
B_\varphi(r)}d_0[hd_1\varphi\wedge (d_0d_1\varphi)^{n-1}],
\end{aligned}\end{equation} where the fifth identity follows from the fact that
$d_1\varphi\wedge d_1\varphi=0.$

Since $d_0[hd_1\varphi\wedge(d_0d_1\varphi)^{n-1}]
=\sum_{i_1}\nabla_{i_10 }\left[h\sum_{j_1,i_2,\ldots}\nabla_{j_11}\varphi\triangle_{i_2j_2}\varphi\ldots\triangle_{i_nj_n}\varphi~\delta^{i_1j_1\ldots
i_nj_n}_{01\ldots(2n-1)}\right]~\Omega_{2n},$ we have
\begin{equation}\label{4.11}\begin{aligned}&\int_{\Omega\backslash
B_\varphi(r)}d_0[hd_1\varphi\wedge (d_0d_1\varphi)^{n-1}]\\=&\int_{\Omega\backslash
B_\varphi(r)}\sum_{i_1}\nabla_{i_10 }\left[h\sum_{j_1,i_2,\ldots}\nabla_{j_11}\varphi\triangle_{i_2j_2}\varphi\ldots\triangle_{i_nj_n}\varphi\delta^{i_1j_1\ldots
i_nj_n}_{01\ldots(2n-1)}\right]~dV\\=&- \int_{S_\varphi(r)}h\sum_{i_1,j_1,\ldots}\nabla_{j_11}\varphi\triangle_{i_2j_2}\varphi\ldots\triangle_{i_nj_n}\varphi\delta^{i_1j_1\ldots
i_nj_n}_{01\ldots(2n-1)}\cdot n_{i_10}~dS,
\end{aligned}\end{equation}by using Stokes' formula again.
Combine (\ref{4.4}) with (\ref{4.11}) to get that $\mu_{\varphi,r}=\triangle_n (\varphi_r)-\chi_{\Omega\backslash
B_\varphi(r)}\triangle_n \varphi$ is given by the right hand side of (\ref{4.2}).
\end{proof}

\begin{thm}\label{t4.1}$($Lelong-Jensen type formula$)$ Let $\Omega$
be a quaternionic strictly pseudoconvex domain. Let $\varphi$ be a continuous
$PSH$ function on $\Omega$ and let $V$ be a locally bounded $PSH$ function on
$\Omega$. Then\begin{equation}\label{4.5}\mu_{\varphi,r}(V)-\int_{B_\varphi(r)}V(\triangle\varphi)^n=\int_{B_\varphi(r)}(r-\varphi)\triangle V\wedge(\triangle\varphi)^{n-1}=\int_{-\infty}^{r}dt\int_{B_\varphi(t)}\triangle V\wedge(\triangle\varphi)^{n-1}.
\end{equation}
\end{thm}
\begin{proof} Firstly assume that $\varphi$ and $V$ are both smooth.
When $r$ is not critical for $\varphi$, apply Stokes-type formula in Lemma \ref{p3.9}, Proposition \ref{p1.1} (1) and Proposition \ref{p2.3} to get
\begin{equation}\label{4.6}\begin{aligned}&\int_{B_\varphi(r)}(r-\varphi)\triangle V\wedge (\triangle \varphi)^{n-1}
=\int_{B_\varphi(r)}(r-\varphi)(-d_1d_0V)\wedge (\triangle \varphi)^{n-1}\\
=&\int_{B_\varphi(r)}-d_1\varphi\wedge d_0V\wedge (\triangle \varphi)^{n-1}=\int_{B_\varphi(r)}d_0V\wedge d_1\varphi \wedge (\triangle \varphi)^{n-1}\\
=&\int_{B_\varphi(r)}-V(d_0d_1\varphi)^{n}+\int_{B_\varphi(r)}d_0[Vd_1\varphi\wedge (\triangle \varphi)^{n-1}]
=\int_{B_\varphi(r)}-V(\triangle\varphi)^{n}+\mu_{\varphi,r}(V),
\end{aligned}\end{equation}
where the last identity follows from (\ref{4.11}). And by
the Fubini theorem to the corresponding measures, we get
\begin{equation*}\int_{B_\varphi(r)}(r-\varphi)\triangle V\wedge(\triangle\varphi)^{n-1}=\int_{-\infty}^{r}dt\int_{B_\varphi(t)}\triangle V\wedge(\triangle\varphi)^{n-1}.
\end{equation*}
Note that for $\varphi,V$ smooth, the map $r\mapsto(\triangle
\varphi_r)^n$ is continuous by Theorem \ref{t3.2}. So the left hand side of the formula (\ref{4.5}),
which coincides with $\int_{\Omega}V(\triangle
\varphi_r)^n-\int_{\Omega}V(\triangle\varphi)^n$, is continuous in
$r$. And it follows from Sard's theorem that almost all values of
$\varphi$ are not critical, so the formula (\ref{4.5}) is also valid
for the critical value $r$ by both integrals in (\ref{4.5})
increasing in $r$ for nonnegative $V$.

If $V$ is smooth and $\varphi$ is merely
continuous, we can find a decreasing sequence
$\{\varphi_{l}\}\subseteq PSH\cap C^{\infty} $ converging to
$\varphi$. Then
$(\triangle \varphi_{l})^n\rightarrow(\triangle \varphi)^n,~(\triangle\varphi_{l,r})^n\rightarrow(\triangle\varphi_{r})^n$
and
$\triangle V\wedge(\triangle\varphi_{l})^{n-1}\rightarrow\triangle V\wedge(\triangle\varphi)^{n-1}$
weakly as $l\rightarrow +\infty$ by Theorem \ref{t3.2}. Apply
(\ref{4.5}) to $\varphi_{l}$ we have
\begin{equation}\begin{aligned}\label{4.7}\int_{\Omega}V(\triangle \varphi_{l,r})^n-\int_{\Omega}V(\triangle\varphi_{l})^n
=&\mu_{\varphi_{l},r}(V)-\int_{\{\varphi_{l}<r\}}V(\triangle\varphi_{l})^n
= \int_{\Omega}\chi_{\{\varphi_{l}<r\}}(r-\varphi_{l})\triangle V\wedge(\triangle\varphi_{l})^{n-1}.
\end{aligned}\end{equation} Let $l\rightarrow +\infty$ in (\ref{4.7}), we get
(\ref{4.5}) for $V$ is smooth and $\varphi$ is merely continuous.\par
Finally, for $V\in PSH\cap L_{loc}^\infty(\Omega)$, let $V_{h}$ be a
decreasing sequence of smooth functions such that
$V=\lim_{h\rightarrow0}V_{h}$ with $V_{h}\in
PSH(\Omega'),~\Omega'\Subset\Omega$. Then
$\triangle V_{h}\wedge(\triangle\varphi)^{n-1}$ converges weakly to
$\triangle V\wedge(\triangle\varphi)^{n-1}$ by Theorem \ref{t3.2}.
And by the monotone
convergence theorem, $\int_{B_\varphi(r)}V_{h}(\triangle\varphi)^n$
converges to $\int_{B_\varphi(r)}V(\triangle\varphi)^n$ and
$\mu_{\varphi,r}(V_{h})$ converges to $\mu_{\varphi,r}(V)$. Apply
(\ref{4.6}) to $V_{h}$ with $\varphi$ continuous and let
$h\rightarrow 0$ to get (\ref{4.5}).
\end{proof}

\begin{appendix}
\section{coincidence of $\triangle_n$ with the quaternionic
Monge-Amp\`{e}re operator}\par

Alesker introduced in \cite{alesker1} the mixed
Monge-Amp\`{e}re operator det$(f_1,\ldots,f_n)$ for
$f_1,\ldots,f_n\in C^2$,\begin{equation*}\text{det}(f_1,\ldots,f_n):=\text{det}~\left(\left(\frac{\partial^2f_1}{\partial q_j\partial \bar{q}_k}(q)\right),\ldots,\left(\frac{\partial^2f_n}{\partial q_j\partial \bar{q}_k}(q)\right)\right),\end{equation*}
where det denotes the mixed discriminant of hyperhermitian matrices. Consider the homogeneous polynomial $\text{det}(\lambda_1A_1+\ldots+\lambda_nA_n)$ in real variables $\lambda_1,\ldots,\lambda_n$ of degree $n$. The coefficient of the monomial $\lambda_1\ldots\lambda_n$ divided by $n!$ is called the mixed discriminant of the matrices $A_i,\ldots,A_n$, and it is denoted by $\text{det}(A_1,\ldots,A_n)$. In particular, when $f_1=\ldots=f_n=f$, $\text{det}(f_1,\ldots,f_n)=\text{det}(f)$. See \cite{alesker1} for more information about the mixed discriminant.\par

\begin{thm}\label{t2.1}Let $f_1,f_2,\ldots,f_n$ be $C^2$ functions in
$\mathbb{H}^n$. Then we have:
\begin{equation}\label{2.17}\triangle_n(f_1,f_2,\ldots,
f_n)=n!~\text{det}~(f_1,f_2,\ldots,f_n).
\end{equation}
\end{thm}
\begin{proof}In the case of $n=1$, by (\ref{2.1000}) we get
\begin{equation}\label{eq:A-laplace}\begin{aligned}\triangle_1f&=\frac{1}{2}\sum_{i,j=0,1}\delta^{ij}_{01}\left(\nabla_{i0 }\nabla_{j1 }f-\nabla_{i1 }\nabla_{j0 }f\right)
=\nabla_{00 }\nabla_{11 }f-\nabla_{01 }\nabla_{10 }f\\
&=(\partial
_{x_0}+\textbf{i}\partial_{x_1})(\partial_{x_0}-\textbf{i}\partial_{x_1})f-(-\partial
_{x_2}-\textbf{i}\partial_{x_3})(\partial _{x_2}-\textbf{i}\partial_{x_3})f=\triangle_{q_1}f,
\end{aligned}
\end{equation}where $\triangle_{q_1}$ is the Laplace operator on $\mathbb{H}$. Now assume that the identity (\ref{2.17})
holds already for $n-1$.\par
Since the linear combination of delta
functions are dense in the space of general functions (see Appendix A in \cite{alesker1} for the proof), it suffices
to prove (\ref{2.17}) in the case of $f_1(q)=\delta_{L}$, where $L$
is a fixed hyperplane $\{\sum_i a_iq_i=0\}$. That is,
\begin{equation}\label{2.18}\triangle_n(\delta_{L},f_2\ldots f_n)(q)=n!~\text{det}(\delta_{L},f_2,\ldots,f_n)(q).
\end{equation}\par
Note that there exists a  unitary matrix $A\in \text{U}_{\mathbb{H}}(n)$ such that the
hyperplane $L=\{\sum_i a_iq_i=0\}=\{q'_1=0\}$, where $q=Aq'$. We
can define new functions $f_i'$ by the formula $f_i'(q'):=f_i(Aq')$,
for $i=1,\ldots,n$. Denote $B'_i=\left(\frac{\partial^2f'_i}{\partial
q'_j\partial
\overline{q}'_k}(q')\right),~~~~~~~B_i=\left(\frac{\partial^2f_i}{\partial
q_j\partial \overline{q_k}}(Aq')\right)$, $i=1,\ldots,n$. Then by
noting that $B'_i=\overline{\left(\frac{\partial^2f'_i}{\partial
\overline{q}'_j\partial q'_k}(q')\right)}$, we can get
$B'_i=\overline{A}B_iA^t$ by (\ref{A*A}). So
\begin{equation*}\lambda_1B_1'+\ldots+\lambda_nB_n'=\overline{A}(\lambda_1B_1+\ldots+\lambda_nB_n)A^t.\end{equation*}
It follows from Theorem 1.1.9 in \cite{alesker1}
that\begin{equation*}\text{det}(\lambda_1B_1'+\ldots+\lambda_nB_n' )=\text{det}(\overline{A}~A^t)\text{det}(\lambda_1B_1+\ldots+\lambda_nB_n)=\text{det}(\lambda_1B_1+\ldots+\lambda_nB_n).
\end{equation*} By definition of the mixed discriminant,
det$(f_1,\ldots,f_n)(Aq')=\text{det}(f_1',\ldots,f_n')(q')$. Hence
\begin{equation*}\begin{aligned}\text{det}(\delta_{\{\sum_i
a_iq_i=0\}},f_2,\ldots,f_n)(q)&=\text{det}(\delta_{\{\sum_i
a_iq_i=0\}},f_2,\ldots,f_n)(Aq')\\
&=\text{det}(\delta_{\{q_1'=0\}},f'_2,\ldots,f'_n)(q').
\end{aligned}\end{equation*}
On the other hand, it follows from Corollary \ref{c2.2} that
\begin{equation*}\triangle_n(\delta_{\{\sum_i
a_iq_i=0\}},f_2,\ldots, f_n)(q)=\triangle_n(\delta_{\{q'_1=0\}},f'_2,\ldots,f'_n)(q').\end{equation*}Therefore it suffices to prove
(\ref{2.18}) in the case $L=\{q_1=0\}$.\par Note that
$\frac{\partial^2\delta_L}{\partial q_j\partial\overline{q}_k}=0$
unless $j=k=1$, that is,
\begin{equation*}\left(\frac{\partial^2\delta_L}{\partial
q_j\partial\overline{q}_k}\right)=\left(
                                    \begin{array}{cccc}
                                      \triangle_{q_1}\delta_L & 0 & \ldots & 0 \\
                                      0 & 0 &   &   \\
                                      \vdots &   & \ddots &   \\
                                      0 &   &   &0 \\
                                    \end{array}
                                  \right),
\end{equation*}where $\triangle_{q_1}$ is the Laplace operator of
$q_1$. By Proposition 1.1 in \cite{alesker1}, we have
\begin{equation}\label{1/n}n~\text{det}(\delta_L,f_2,\ldots,f_n)=\triangle_{q_1}\delta_L\text{det}(C_2,\ldots,C_n),\end{equation}where
$C_i$ is the $(n-1)$ matrix $\left(\frac{\partial^2f_i}{\partial
q_j\partial\overline{q}_k}\right)_{j,k=2}^n$, $i=2,\ldots,n$. Then by (\ref{1/n}) we have
\begin{equation*}\int_{\mathbb{H}^n} h~\text{det}(\delta_{L},f_2,\ldots,f_n)=\int_{\mathbb{H}^n} h\frac{1}{n}\triangle_{q_1}\delta_L~\text{det}(C_2,\ldots,C_n)
=\int_{L} \frac{1}{n}\triangle_{q_1}\left[h~\text{det}(C_2,\ldots,C_n)\right]|_{q_1=0},
\end{equation*}for any test function $h$. Here   integrals concerning  generalized functions are as in \cite{alesker1}. It is easy to see that the last integral depends on not
more than $2$-order drivatives of $f_2,\ldots,f_n$ in the direction
$q_1$. Thus, for $i=2,\ldots,n$, we can assume that
$f_i(q_1,\ldots,q_n)=p_i(q_1)\widehat{f}_i(q_2,\ldots,q_n),$ where
$p_i(q_1)$ is a polynomial not more than 2-order depending only on
$q_1$, and $\widehat{f}_i(q_2,\ldots,q_n)$ depends only on
$q_2,\ldots,q_n$.\par Denote by $\widehat{C_i}$ the
$(n-1)$-matrix:$\left(\frac{\partial^2\widehat{f}_i(q_2,\ldots,q_n)}{\partial
q_j\partial \overline{q}_k}\right)_{j,k=2}^n$, $i=2,\ldots,n$. It
follows that
\begin{equation*}\text{det}(C_2,\ldots,C_n)=p_2(q_1)\ldots
p_n(q_1)\text{det}(\widehat{C_2},\ldots,\widehat{C_n})=p_2(q_1)\ldots
p_n(q_1)\text{det}(\widehat{f}_2,\ldots,\widehat{f}_n).\end{equation*} So we
have
\begin{equation}\label{2.20}\begin{aligned}\int_{\mathbb{H}^n} h~\text{det}(\delta_{L},f_2,\ldots,f_n)(q)
=&\int_{L} \frac{1}{n}\triangle_{q_1}\left[h~p_2(q_1)\ldots
p_n(q_1)\right]|_{q_1=0}\text{det}(\widehat{f}_2,\ldots,\widehat{f}_n)\\
=&\int_{L} \frac{1}{n!}\triangle_{q_1}\left[h~p_2(q_1)\ldots
p_n(q_1)\right]|_{q_1=0}\triangle_{n-1}(\widehat{f}_2,\ldots,\widehat{f}_n),
\end{aligned}\end{equation}
where the last identity follows from our inductive assumption.\par On
the other hand, by noting that $\triangle_{ij}\delta_{L}=0$ unless
$\{i,j\}=\{0,1\}$, we have
\begin{equation*}\begin{aligned}\int_{\mathbb{H}^n} h~\triangle_n(\delta_{L},f_2,\ldots,f_n)&=\int_{\mathbb{H}^n} h~\sum_{i_1,j_1,\ldots}\delta^{i_1j_1\ldots
i_nj_n}_{01\ldots(2n-1)}\triangle_{i_1j_1}\delta_L\triangle_{i_2j_2}f_2\ldots\triangle_{i_nj_n}f_n\\&=\int_{\mathbb{H}^n} h~2\sum_{i_2,j_2,\ldots}\delta^{01i_2j_2\ldots
i_nj_n}_{0123\ldots(2n-1)}\triangle_{01}\delta_L\triangle_{i_2j_2}f_2\ldots\triangle_{i_nj_n}f_n\\
 &=\int_{\mathbb{H}^n} h~\triangle_{q_1}\delta_L~p_2(q_1)\ldots
p_n(q_1)\sum_{i_2,j_2,\ldots }\delta^{i_2j_2\ldots
i_nj_n}_{23\ldots(2n-1)}\triangle_{i_2j_2}\widehat{f}_2\ldots\triangle_{i_nj_n}\widehat{f}_n\\
&=\int_{L} \triangle_{q_1}\left[h~p_2(q_1)\ldots
p_n(q_1)\right]|_{q_1=0}\sum_{i_2,j_2,\ldots }\delta^{i_2j_2\ldots
i_nj_n}_{23\ldots(2n-1)}\triangle_{i_2j_2}\widehat{f}_2\ldots\triangle_{i_nj_n}\widehat{f}_n\\
&=\int_{L} \triangle_{q_1}\left[h~p_2(q_1)\ldots
p_n(q_1)\right]|_{q_1=0}\triangle_{n-1}(\widehat{f}_2,\ldots,\widehat{f}_n)\\
&=n!\int_{\mathbb{H}^n} h~\text{det}(\delta_{L},f_2,\ldots,f_n)(q),
\end{aligned}
\end{equation*}where the last identity follows from (\ref{2.20}).
\end{proof}
\begin{cor}\label{t2.2}The operator $f\mapsto\triangle_nf$
coincides with the quaternionic Monge-Amp\`{e}re operator.
\end{cor}

\end{appendix}

 \end{document}